\documentclass[10 pt]{amsart}

\usepackage{amsmath}
\usepackage{amsfonts}
\usepackage{amssymb}
\usepackage{amstext}
\usepackage{amsbsy}
\usepackage{amsopn}
\usepackage{amsthm}
\usepackage{amsxtra}
\usepackage{graphicx}
\usepackage{hyperref}
\usepackage{enumitem}
\usepackage{xcolor}
\usepackage[normalem]{ulem}
\usepackage{soul}
\usepackage{tikz}
\usepackage[justification=centering]{caption}

\newtheorem{theorem}{Theorem}[section]
\newtheorem{lemma}[theorem]{Lemma}
\newtheorem{proposition}[theorem]{Proposition}

\newtheorem{corollary}[theorem]{Corollary}
\theoremstyle{definition}
\newtheorem{definition}[theorem]{Definition}
\newtheorem{question}[theorem]{Question}
\newtheorem{conjecture}[theorem]{Conjecture}
\newtheorem{example}[theorem]{Example}
\newtheorem{remark}[theorem]{Remark}

\newtheorem*{comment*}{Comment}

\newcommand{\R}{\mathbb{R}}
\newcommand{\Z}{\mathbb{Z}}

\newcommand{\N}{\mathbb{N}}

\newcommand{\T}{\mathbb{T}}

\newcommand{\CA}{\mathcal{A}}

\newcommand{\CL}{\mathcal{L}}

\newcommand{\CS}{\mathcal{S}}

\newcommand{\aut}{\mathrm{Aut}}

\newcommand{\id}{\mathrm{Id}}

\newcommand{\tr}{\mathrm{tr}}
\newcommand{\diam}{\mathrm{diam}}

\newcommand{\interior}[1]{%
  {\kern0pt#1}^{\mathrm{o}}%
}

\begin{document}
\title{Chaotic almost minimal actions}
\author{Van Cyr}
\author{Bryna Kra}
\author{Scott Schmieding}
\thanks{BK was partially supported by the National Science Foundation grant DMS-2348315 and SS by the National Science Foundation grant DMS-2247553.}
\maketitle
\begin{abstract}
Motivated by Furstenberg's Theorem on sets in the circle invariant under multiplication by a non-lacunary semigroup, we define a general class of dynamical systems possessing similar topological dynamical properties.
We call such systems  chaotic almost minimal, reflecting that these systems are chaotic, but in some sense are close to minimal. We study properties of the acting group needed to admit such an action, and show the existence of a chaotic almost minimal $\mathbb{Z}$-action. We show there exists chaotic almost minimal $\mathbb{Z}^{d}$-actions which support multiple distinct nonatomic ergodic probability measures.
\end{abstract}

\section{Furstenberg's Theorem}
As celebrated result of Furstenberg~\cite{furstenberg-disjointness} shows that the action of a pair of commuting maps $x\mapsto px\mod 1$ and $x\mapsto qx\mod 1$ on  the circle $\T = \R/\Z$, where $p, q\geq 2$ are multiplicatively independent integers, is highly constrained.  Namely, any closed set that is invariant under multiplication by such $p$ and by $q$
is either finite or is all of $\T$.
Here, we explore what dynamical behavior is behind this type of dichotomy.

We focus on two properties: that there is a dense set of points with finite orbit, and that finite orbits are the the only proper closed invariant sets.
This leads us to define a faithful and transitive  action of a group
(or semigroup)
on a compact metric space $X$  to be {\em chaotic almost minimal} if  there is a dense set of points with finite orbit
and every proper, closed invariant subset of $X$ is finite (see Definition~\ref{def:CAM} for the precise conditions).
The motivation for the name is that a transitive system with dense periodic points is chaotic, for example in the sense of~\cite{devaney} (see also~\cite{ BBCDS, GW}), and the condition on the proper, closed invariant subsets is termed almost minimal in the context of algebraic actions by Schmidt~\cite{schmidt}.

Extending the definition, we call a group (or semigroup) {\em
chaotic almost minimal} if it admits a chaotic almost minimal action on some compact metric space.
In this terminology, Furstenberg's Theorem shows that $\N^2$
is chaotic almost minimal, and the invertible symbolic cover of this system, such as that used by Rudolph~\cite{Rudolph}, shows that $\Z^2$ is chaotic almost minimal.
Similarly one can show that $\N^d$ and $\Z^d$ are chaotic almost minimal for any $d\geq 2$.  While we focus on discrete groups,
we can extend these notions to continuous actions,
defining an action of a locally compact group $G$ to be {\em chaotic almost minimal} if there is a dense set of points whose stabilizer is a lattice in $G$ and every proper compact $G$-invariant subset is the quotient of $G$ by a lattice.
In this setting,
there are also well-known examples of continuous actions which are chaotic almost minimal.
For instance,
Ratner's Theorem~\cite{ratner} implies the action of the unipotent flow on the space $\textrm{SL}_{2}(\mathbb{R})/\textrm{SL}_{2}(\mathbb{Z})$ is chaotic almost minimal. By taking the one-point compactification,
we obtain a chaotic almost minimal $\mathbb{R}$-action on a compact space.  We note however that for any time $t > 0$, the time $t$-map of the flow $\varphi_{t}$ on $SL_{2}(\mathbb{R})/SL_{2}(\mathbb{Z})$ is \emph{not} chaotic almost minimal according to our definition: since there are periodic trajectories of arbitrary length, for any time $t > 0$, there are invariant circles on which the time $t$ map acts by irrational rotation.

Several other examples of chaotic minimal groups have been well studied, including ${\rm SL}_d(\Z)$ for $d\geq 2$ and the automorphism group of a shift of finite type, and we discuss  some of these in Section~\ref{sec:examples}.
All of these examples, and the  variety of other groups that we show have chaotic almost minimal actions (such as those following from the results in Section~\ref{sec:residually-finite}) are all either higher rank or are complicated in some other way.
It is therefore natural to ask about the existence of a $\mathbb{Z}$-action that is chaotic almost minimal. At the same time, with an eye toward the well-known question of Furstenberg on measures invariant under multiplication by two multiplicatively independent integers, it is also natural to wonder whether being chaotic almost minimal is compatible with having multiple nonatomic, ergodic probability measures.

We show that not only do chaotic almost minimal $\mathbb{Z}$-systems exist, but that such systems are capable of supporting more than one nonatomic, ergodic probability measure.  Specifically, we  construct a symbolic chaotic almost minimal $\mathbb{Z}$-system with (at least) two distinct, nonatomic, ergodic probability measures.

\begin{theorem}
\label{th:Z-CAM}
There exists a topological $\mathbb{Z}$-system that is chaotic almost minimal. Furthermore, this system supports two distinct nonatomic ergodic probability measures.
\end{theorem}

Our definition of CAM systems is strongly motivated by Furstenberg's Theorem on $\times p, \times q$ invariant subsets of the torus. Moreover, we are also motivated here by Furstenberg's question regarding the possible nonatomic ergodic measures for $\times p, \times q$ on the circle. A consequence of Theorem~\ref{th:Z-CAM} is that if the answer to Furstenberg's question is positive, meaning there is a unique nonatomic ergodic measure invariant under a non-lacunary semigroup, then the CAM structure alone is not responsible for   this uniqueness. Moreover, in higher ranks the same result holds, and we prove that in fact, there exists $\mathbb{Z}^{d}$-CAM systems possessing two distinct nonatomic ergodic measures.

\begin{theorem}
\label{th:Zd-CAM}
For every $d \ge 1$, there exists a topological $\mathbb{Z}^{d}$-system that is chaotic almost minimal and supports two distinct nonatomic ergodic probability measures.
\end{theorem}
%

As in the one dimensional case of Theorem~\ref{th:Z-CAM}, our constructed system is a subshift over the alphabet $\{0,1\}$.

The details of the construction of the CAM $\mathbb{Z}$-system are in Section~\ref{sec:construction}. We show that the resulting system is chaotic almost minimal in Proposition~\ref{prop:Z-CAM}, and in Section~\ref{sec:measures} we prove that this system has more than one nonatomic, ergodic measure. In Section~\ref{sec:higher}, we adapt this construction to a $\mathbb{Z}^{d}$-action, and show the existence of a $d$-dimensional subshift that is chaotic almost minimal and supports multiple distinct nonatomic ergodic probability measures.

A property shared by all of the groups that we show are chaotic almost minimal is that they are all residually finite, and in fact it is not hard to show that any group which admits a chaotic action must be residually finite (see Section~\ref{sec:residually-finite}).  This leaves open the question of whether every residually finite group gives rise to a chaotic almost minimal action, and we conjecture (see Section~\ref{sec:further}) that this holds.

We comment that we define chaotic almost minimal systems in terms of invariant subsets and not in terms of orbits of points.  As a tool to showing the existence of chaotic almost minimal systems, we often proceed by showing a weaker property, which we call {\em weakly chaotic almost minimal}, in which we replace the condition on subsets by the condition that all orbits are either finite or dense.
While every chaotic almost minimal system is weakly chaotic almost minimal, the converse does not hold and a construction showing that the two notions are not the same is given in Example~\ref{ex:weak-not-strong}.  However, in Proposition~\ref{prop:weak-to-strong} we show these notions are the same for expansive systems, which is why we are able to make use of the weakly chaotic almost minimal property in our constructions.  A related question, along with some others and a conjecture, are discussed in Section~\ref{sec:further}.

\subsection*{Acknowledgment} We thank Xiangdong Ye for helpful remarks in the preparation of this article.

\section{CAM systems}
\subsection{General definitions and notation}
For a countable group $G$ and compact metric space $X$, a {\em topological $G$-system $(X,T)$} is an action of $G$ on $X$ by homeomorphisms, meaning a homomorphism $T \colon G \to \textnormal{Homeo}(X)$. Given such a system and $g \in G$ we write $T_{g} \colon X \to X$ for the action by the element $g$ on $X$.
We use the same terminology when $G$ is a semigroup, referring to a topological $G$-system $(X, T)$, but assuming that the associated actions $T_g\colon X\to X$ are continuous maps which are not necessarily invertible. When $G = \mathbb{Z}$ we write simply $(X,T)$ to mean $T$ is a homeomorphism from $X$ to itself.

The system $(Y, S)$ is a {\em (topological) factor} of the system $(X, T)$ if there exists a continuous surjective map $\pi\colon X\to Y$ such that $\pi\circ T = S\circ \pi$.

The $G$-system $(X, T)$ is {\em topologically transitive} if for all open sets $U, V\subset X$, there is some $g\in G$ such that $T_gU\cap V\neq\emptyset$, and is {\em point transitive} if there exists some $x\in X$ such that the orbit $\{gx:g\in G\}$ is dense in $X$.

Some of our examples and constructions involve symbolic systems.  For a finite set $\CA$, called the {\em alphabet}, we consider the compact metric space $\CA^\Z$, writing  $x\in\CA^\Z$ as
$x = (x_n)_{n\in\Z}$, where $\CA^\Z$ is endowed with the metric
$d(x, y) = 2^{-\{\inf|n|: x_n\neq y_n\}}$. The shift $\sigma\colon \CA^\Z \to\CA^\Z$ is defined
by $(\sigma x)_n = x_{n+1}$ for all $n\in\Z$.
A {\em subshift} is a system $(X, \sigma)$, where $X\subset\CA^\Z$ is a compact $\sigma$-invariant subset and $\mathcal{A}$ is some finite alphabet.

For $n\geq 0$, a concatenation $w =w_0\dots w_{n-1}\in\CA^n$ is called a {\em word} $w$ and we say that this word has {\em length} $n$. Given a subshift $(X, \sigma)$ over the alphabet $\mathcal{A}$ and word $w \in \mathcal{A}^{n}$ , the {\em cylinder set} $C_{k}(w)$ is the collection of all $x\in X$ such that $x_k = w_0, \dots x_{k+n-1} = w_n$. By a cylinder set in $(X,\sigma)$ we mean a set $C_{k}(w)$ for some $k \in \mathbb{Z}$ and word $w$ over the respective alphabet. A word $w$ is in the  {\em language} $\CL(X)$ of $(X, \sigma)$ if there is some nonempty cylinder set determined by $w$.  The language $\CL(X) = \bigcup_{n\geq 0}\CL_n(X)$, where  $\CL_n(X)$ denotes all of the words of length $n$ in $\mathcal{L}(X)$.

\subsection{Definition of CAM systems}
We now define our main object of study. All groups and semigroups here are assumed to be discrete.
\begin{definition}
\label{def:CAM}
A topological $G$-system $(X,T)$ is
{\em  chaotic almost minimal (CAM)} if all of the following hold:
\begin{enumerate}
\item
The system $(X,T)$ is topologically transitive and the action of $T$ on $X$ is faithful.
\item
The space $X$ contains a dense set of
points whose $T$-orbit is finite.
\item
Every proper closed $T$-invariant
subset
of $X$ is finite.
\end{enumerate}
\end{definition}
When $(X,T)$ satisfies these conditions, we say that $(X, T)$ is a {\em CAM system}, or shorten this further and say that it is {\em CAM}.  When we want to emphasize the group acting, we call this a $G$-CAM system.

We note that for general groups, we can easily construct examples showing that no two of these conditions implies the third.  There are various variations of these assumptions possible, and one which is useful for our purposes is a weakening of the third condition to be about points, and not invariant sets (see Section~\ref{sec:weaklyCAM}).
We further note that if a system is CAM, then it is topologically transitive, has a dense set of points with finite orbit, and the only compact minimal subsystems are finite. However, the converse is not true in general, as can be seen from Example~\ref{ex:weak-not-strong}.  Related examples are constructed in~\cite{DY}, where they produce an $\N$-system with dense periodic points such that every point either is periodic or has dense orbit.

We also study which groups admit CAM actions, and this is captured in the next definition.
\begin{definition}
The group or semigroup $G$ is {\em (topologically)  CAM} if there exists an infinite compact metric space $X$ such that the system $(X, G)$ is a CAM system.
\end{definition}

\subsection{Classical examples of systems that are CAM}
\label{sec:examples}

Before proceeding further, we give several examples of CAM systems.

\begin{example}[Systems invariant under a non-lacunary (semi-) group]
\label{ex:Furstenberg}
Our motivating example is
Furstenberg's Theorem~\cite{furstenberg-disjointness}, which proves that $\times p,\times q$ for multiplicatively independent integers $p, q \geq 2$ acting on the torus is an $\N^2$-CAM system.  A symbolic $\Z^2$-CAM system is given by Rudolph's  coding~\cite{Rudolph} of of this system. By taking the restriction of this system to $\mathbb{N}^2$, we also obtain an example of a symbolic $\N^2$-CAM system.
\end{example}

Generalizations of this example for actions on higher dimensional tori  that are CAM are given in Berend~\cite{berend}. Another CAM action on higher dimensional tori is given in the next example.

\begin{example}
    [The action of ${\rm SL}_d(\Z)$ on $\T^d$]
\label{example:SLdZ}
We make use of results in the literature to check that the action of ${\rm SL}_d(\Z)$ on $\T^d$ is CAM.
Several papers prove that any infinite set has dense orbit (see for example~\cite{GS, muchnik, Dong1}). It suffices then to check that periodic points are dense, but this is trivial.  Namely, if  $v \in \mathbb{Q}^{d} / \mathbb{Z}^{d}$, write $v$ using a common denominator for all coordinates. Then for any $A \in {\rm SL}_{d}(\mathbb{Z})$ the denominators of the coordinates of $Av$ are again bounded above by the original denominators, and so the  orbit of $v$ must be finite.
We note that any probability measure invariant under this action is a linear combination of Haar measure and atomic measures (this is a simple case of the deep results in~\cite{BQ, BFLM}).
These results cover more general settings, as shown in~\cite{GS, muchnik,BQ,BFLM}.
\end{example}

A different type of example arises by considering the symmetries of symbolic systems.

\begin{example}[The automorphism group of a mixing shift of finite type]
\label{example:autos}
Let $(X, T)$ be an infinite mixing $\mathbb{Z}$-shift of finite type and let $\aut(X,T)$ denote the group of automorphisms of $(X,T)$, meaning the group of self-homeomorphisms of $X$ that commute with $T$. Then the action of $\aut(X,T)$ on $X$ is CAM. Indeed, the density of periodic points is immediate, as the action of $\aut(X,T)$ preserves $T$-periodic points of a given period and $T$-periodic points are dense in $X$.  If $Y$ is an infinite, closed $\aut(X,T)$-invariant set, then in particular it is $T$-invariant, and so $Y$ contains an aperiodic point.
It follows from~\cite[Theorem 9.2]{BLR} that the $\aut(X,T)$-orbit of this aperiodic point is dense in $X$, and so the system is weakly CAM.  Since it is expansive, it follows from Proposition~\ref{prop:weak-to-strong} that the system is CAM.
It is interesting to note that it is also proved in~\cite{BLR} that the only nonatomic measure on $X$ invariant under $\aut(X,T)$ is the measure of maximal entropy for $T$.
\end{example}

\subsection{Basic dynamical properties of CAM systems}

For any group (or semigroup) $G$, the collection of  $G$-CAM systems is closed under factors, orbit equivalence, and flow equivalence, but is not closed under taking products. We verify these and other simple properties of CAM systems.

\begin{proposition}\label{prop:FurstUnctble}
If $(X, T)$ is infinite, has dense periodic points, and is point  transitive,
then $X$ has no isolated points. 
In particular, if $(X,T)$ is CAM then $X$ is perfect and uncountable.   
\end{proposition}

\begin{proof}
We proceed by contradiction and assume that $x\in X$ is an isolated point.  Since the periodic points are dense in $X$, it follows that $x$ itself is periodic.  If $(U_n)_{n\in\N}$ is a countable base of open sets for $X$, then for each $n\in\N$, the set $V_n = \bigcup_{m> 0} T^{-m}U_n$ is open and dense.  The intersection $\bigcap_{n\in\N}V_n$ is a residual set and by transitivity, every point in it has  forward dense orbit.   In particular, there is some point with a dense forward orbit and so the periodic point $x$ itself has a dense forward orbit, contradicting the assumption that the system is infinite. Thus $X$ has no isolated points. Since $X$ is compact, it follows that it is perfect and uncountable.
\end{proof}

\begin{proposition}
A system  with dense periodic points that is point transitive  admits a Borel invariant probability measure of full support.
\end{proposition}
Again, this applies immediately to any CAM system.
\begin{proof}
Let $(O_{i})_{i\in\N}$ be a collection of disjoint periodic orbits  whose union is dense in the space and let $(m_{i})_{i\in\N}$ be a sequence of positive real numbers such that $\sum_{i\in\N} m_i = 1$.   Setting $\mu = \sum_{i} \frac{m_{i}}{|O_{i}|}\left(\sum_{x \in O_{i}}\delta_{x}\right)$, where $\delta_x$ denotes the Dirac measure at $x$, we produce such a measure.
\end{proof}

\begin{proposition}
\label{prop:measures}
Any nonatomic invariant Borel probability measure on a CAM system has full support.
\end{proposition}

\begin{proof}
    As the support of the measure is a closed, invariant set, any  nonatomic measure has some infinite closed invariant set.  Invariance implies that the measure has full support.
\end{proof}


\begin{proposition}
For any factor of a CAM system, if the group action on the factor is faithful, then the factor is also a CAM system.  In particular, any infinite factor of a $\Z$-CAM system is $\Z$-CAM.
\end{proposition}
\begin{proof}
Suppose the system $(X, T)$ is CAM and let $\varphi\colon X\to Y$ be a factor map onto some system $(Y,S)$. Clearly $S$-periodic points are dense in $Y$, since $T$-periodic points are dense in $X$, and $(Y,S)$ is transitive since $(X,T)$ is. If $A \subset Y$ is an infinite closed $S$-invariant subset, then  $\pi^{-1}(A)$ is an infinite closed $T$-invariant subset of $X$.  Since the system $(X, T)$ is CAM, it follows that
$\pi^{-1}(A) = X$ and so we have that $ A = Y$.  Thus when the action $S$ on $Y$ is faithful, we have that $(Y,S)$ is CAM.

The conclusion  about $\Z$-CAM factors follows from the fact that the action on an infinite factor of a $\Z$-system is  automatically faithful.
\end{proof}

The next result follows quickly from the definitions (see for example~\cite{putnam} for background).
\begin{proposition}
Any system that is topologically orbit equivalent to a CAM system is also CAM.\end{proposition}

\begin{proof}
The first two properties of Definition~\ref{def:CAM} are immediate. Suppose $(X,S)$ and $(Y,T)$ are topologically orbit equivalent $G$-systems via an orbit equivalence $h \colon X \to Y$.  If the system $(X,S)$ is not CAM, then there exists a proper infinite $G$-invariant closed subset $Z \subset X$. But then $h(Z)$ would is a proper infinite $G$-invariant closed subset of $Y$, and so $(Y,T)$ is not CAM.
\end{proof}

To close this section, we use results of~\cite{shi-xu-yu} we derive a result for  $\Z$-CAM systems (see Remark~\ref{remark:no-N} for contrasting behavior).
\begin{proposition}
If $(X,T)$ is an expansive $\Z$-CAM system, then $(X, T)$ is a subshift.
\end{proposition}
\begin{proof}
    Since the system $(X,T)$ is a $\Z$-CAM system, every point is either transitive or is periodic.  In particular, every point is recurrent.  By~\cite[Theorem 1.2]{shi-xu-yu}, it follows that $(X,T)$ is a subshift.
\end{proof}

\subsection{Many {$\mathbb{Z}$}-CAM systems}
\label{sec:more-Z}
In section~\ref{sec:Zcamsystem} we prove that there exists a  $\mathbb{Z}$-CAM subshift. We use this here to show that the existence of one such subshift implies that  CAM $\mathbb{Z}$-subshifts are abundant in a certain sense. In fact, for any $n \ge 2$, they are dense in the space of infinite totally transitive subsystems of the full shift $(X_{n},\sigma_{n})$ with the Hausdorff metric. They are not however a generic set, as~\cite[Theorem 1.4]{PS} shows there is a generic set of minimal subshifts in the space of totally transitive subshifts.

\begin{theorem}\label{thm:embedCAM}
If $(Y,\sigma_{Y})$ is an infinite mixing shift of finite type, then $(Y,\sigma_{Y})$ contains a CAM subshift.
\end{theorem}

\begin{proof}
Let $(X,\sigma)$ be the $\Z$-CAM shift constructed in Section~\ref{sec:Zcamsystem}, and note that $(X,\sigma)$ is contained in the full 2-shift $(X_{2},\sigma_{2})$.
Given $m \ge 1$, let $(X^{(m)},\sigma^{(m)})$ denote the height $m$ discrete tower over $(X,\sigma)$, meaning that  $X^{(m)} = X \times \{0,\ldots,m-1\}$ and
\begin{equation*}
\sigma^{(m)}(x,i) =
\begin{cases}
(x,i+1) \textrm{ mod } m & \textrm{if } i< m-1\\
(\sigma(x),0) & \textrm{if } i=m-1.
\end{cases}
\end{equation*}
Since $(X,\sigma)$ is CAM, it is straightforward to check that $(X^{(m)},\sigma^{(m)})$ is also CAM.

Given an infinite mixing shift of finite type $(X, \sigma_Y$), we show  that
there exists $m\geq 1$ such that $(X_{2}^{(m)},\sigma_{2}^{(m)})$ embeds into $(Y,\sigma_{Y})$. Since $(X,\sigma)$ is contained in $(X_{2},\sigma_{2})$, once we have this result the theorem follows.

Let $q_{n}(Z)$ denotes the points of least period $n$ in the subshift $Z$.
By Krieger's Embedding Theorem~\cite{krieger}, it suffices to show there exists $m\geq 1$ such that both of the following hold:
\begin{enumerate}
\item
$h(\sigma_{2}^{(m)}) < h(\sigma_{Y})$.
\item
$q_{n}(X_{2}^{(m)}) \le q_{n}(Y)$ for all $n \ge 1$.
\end{enumerate}

For the entropy condition, note that
$h(\sigma_{2}^{(m)}) = \frac{1}{m}h(\sigma_{2})$ for all $m\geq 1$, and so by choosing $m$ sufficiently large we can guarantee that
$h(\sigma_{2}^{(m)}) < h(\sigma_Y)$.  Thus it suffices to find $m$ satisfying the periodic point condition.



To show this, we recall some notation and results (see for example~\cite[Section 10.1]{LM}).  Let $A$ be a $\mathbb{Z}_{+}$-matrix presenting the shift $(Y,\sigma_{Y})$.
The function $q_n$ satisfies the formula
\begin{equation}
    \label{eq:qn}
q_{n}(Y) = \tr_{n}(A) = \sum_{d \mid n}\mu(\frac{n}{d})\tr(A^{d}),
\end{equation}
where $\mu$ denotes the M\"{o}bius function and $\tr$ the trace. Letting $s^{\times}(A)$ denote the nonzero spectrum of $A$ (excluding eigenvalues of modulus one) and $\lambda_{A}$ denote the Perron-Frobenius eigenvalue for $A$, then  for any $mn > 1$, we have
$$\tr_{mn}(A) \ge \lambda_{A}^{mn} - \frac{\lambda_{A}^{mn/2+1} - \lambda_{A}}{\lambda_{A}-1}
- \sum_{\mu \in s^{\times}(A), \mu \ne \lambda_{A}}\frac{|\mu|^{mn+1}-|\mu|}{|\mu|-1}.$$
Let $S = \max\{|\mu| : \mu \in s^{\times}(A), \mu \ne \lambda_{A}\}$ and $K_{1} = \frac{\lambda_{A}}{\lambda_{A}-1}$. It follows that
$$\tr_{mn}(A) \ge \lambda_{A}^{mn}-K_{1}\left(\lambda_{A}^{mn/2}-1\right)-K_{2}(S^{mn}-1)$$
for some constant $K_{2}$ which does not depend on either $m$ or $n$. By the Perron-Frobenius Theorem, we have that  $S/\lambda_{A} < 1$ and so
$$\frac{1}{\lambda_{A}^{mn}}\tr_{mn}(A) \ge 1 - K_{1}(\lambda_{A}^{-mn/2}-\lambda_{A}^{-mn}) - K_{2}\left[\left(\frac{S}{\lambda}\right)^{mn}-\lambda_{A}^{-mn}\right].$$
Combining this with~\eqref{eq:qn},
there is some $m_{1}\geq 1$ such that for all $m \ge m_{1}$ and all $n \ge 1$,
\begin{equation}
    \label{eq:lower-bound}
\frac{1}{\lambda_{A}^{m_{1}n}}q_{m_{1}n}(Y) \ge \frac{2}{3}.
\end{equation}
Considering the full $2$-shift $(X_{2},\sigma_{2})$, we have that
$$q_{mn}(X_{2}^{(m)}) = m q_{n}(X^{(2)}) \le m2^{n}.$$
Thus for all $m,n\geq 1$
$$\frac{1}{\lambda_{A}^{mn}} q_{mn}(X_{2}^{(m)}) \le m\left( \frac{2}{\lambda_{A}^{m}} \right)^{n}.$$
Taking $m_{2}$ such that $\frac{2}{\lambda_{A}^{m}} < \frac{1}{3m}$ for all $m \ge m_{2}$, then for for all $ m\geq m_2$ and any $n\geq 1$, we have that
$$\frac{1}{\lambda_{A}^{mn}} q_{mn}(X_{2}^{(m)}) \le m\left( \frac{2}{\lambda_{A}^{m}} \right)^{n} < m\left( \frac{1}{3m} \right)^{n} < 2/3.$$
Taking $M = \max\{m_{1},m_{2}\}$ and combining this with~\eqref{eq:lower-bound},
we have that for all $n\geq 1$,
$$\frac{1}{\lambda_{A}^{Mn}}q_{Mn}(Y) \ge \frac{2}{3} > \frac{1}{\lambda_{A}^{Mn}}q_{Mn}(X_{2}^{(M)}).$$
Thus it follows that $q_{Mn}(Y) \ge q_{Mn}(X_{2}^{(M)})$,
verifying the necessary periodic point condition for the application of Krieger's Embedding Theorem.
\end{proof}
For a full shift $(X_{n},\sigma_{n})$, let $\mathcal{S}(X_{n})$ be the space of all subshifts of $X_{n}$, meaning that  $\mathcal{S}(X_{n})$ is the set of all compact $\sigma_{n}$-invariant subsets of $X_{n}$ endowed with the Hausdorff metric (more background on these spaces can be found in~\cite{PS}). We  use Theorem~\ref{thm:embedCAM} to show that for any $n \ge 2$, the closure of the set of CAM shifts in $\mathcal{S}(X_{n})$ contains all infinite totally transitive subshifts.
\begin{theorem}
Fix $n \ge 2$ and let $(W,\sigma_{W})$ be an infinite totally transitive subshift in $X_{n}$. Then $W$ is a limit of $\Z$-CAM subshifts in $\mathcal{S}(X_{n})$.
\end{theorem}
\begin{proof}
Fix $m \ge 1$. Since $W$ is infinite and totally transitive, using standard Markov approximation there exists an infinite mixing shift of finite type $(Z,\sigma_{Z})$ such that $d_{H}(Z,W) < 2^{-m}$ where $d_{H}$ denotes the Hausdorff distance. This means that $Z$ and $W$ have the same $m$-languages. Using the proof of~\cite[Theorem 6.4]{PS}, it follows that there is an infinite mixing shift of finite type $Y \subset Z$ such that every point in $Y$ contains every word of length $m$ from $W$.
In particular, the systems $Y$ and $W$ have the same $m$-languages. Applying Theorem~\ref{thm:embedCAM} to the system $Y$,
we have that $Y$ contains a CAM subshift $X$ such that the $m$-languages of $X$ and $W$ agree, and hence $d_{H}(X,W) < 2^{-m}$.
\end{proof}
We note that we can not directly apply Theorem~\ref{thm:embedCAM}; while it shows that we can embed a CAM shift into $Z$, it does not a priori guarantee that the embedded CAM is close in the Hausdorff metric to $Z$.

\subsection{Dynamical restrictions on CAM actions}
\label{sec:dynamical-restrictions}
\begin{theorem}
If $X$ is locally connected compact metric space, then $X$ does not support an expansive $\mathbb{Z}$-CAM action.
\end{theorem}
Our argument closely follows the one used by Ma\~{n}\'{e}~\cite{Mane}.
\begin{proof}
 Suppose $(X,T)$ is CAM. Fix some periodic points $x,y \in X$ and choose $\varepsilon > 0 $ such that $\bigcup_{i=1}^{\ell}T^{i}(B_{\varepsilon}(x))$ is a proper subset of $X$ and $y \not \in B_{2\varepsilon}(x)$; such points exist since $(X,T)$ is CAM. Without loss of generality, we can assume that both $x$ and $y$ have period $\ell$.

Consider the $\varepsilon$-stable set of $x$, defined by
$$W_\varepsilon^s(x) = \{
y\in X: d(T^nx, T^ny)\leq\varepsilon \text{ for all } n\geq 0
\},$$
and let $CW_{\varepsilon}^{s}(x)$ denote the connected component of $x$ in $W_{\varepsilon}^{s}$.
By~\cite[Proposition C]{Hiraide1990}, making $\varepsilon > 0$ smaller if necessary, there exists $\delta > 0$ such that for all $x \in X$, we have
$\diam(CW_{\varepsilon}^{s}(x)) \ge \delta$.

Given such $\delta$, by~\cite[Corollary 2.4]{Kato1993}, there exists $M > 0$ such that for all $m \ge M$, if $\diam(CW_{\varepsilon}^{s}(x)) \ge \delta$, then $\textrm{diam}(T^{-m}(CW_{\varepsilon}^{s}(x)) \ge 4 \delta$. Choose such an $m$ and without loss of generality, assume $\delta$ is sufficiently small such that $B_{\delta}(y) \cap B_{\varepsilon}(x) = \emptyset$ and $\bigcup_{i=1}^{\ell}T^{i}(B_{\delta}(y))$ is proper in $X$. Moreover we may assume that $\ell$ and $m$ are relatively prime.  It follows there exists two disjoint connected components $A^{1}_{1}, A^{1}_{2}$ of $T^{-m}(CW_{\varepsilon}^{s}(x)) \setminus B_{\delta/4}(y)$ for which $\diam(A_{i}^{1}) = \delta$ for $i=1,2$.
We now apply the same procedure to both $A^{1}_1$ and $A^1_2$ to obtain four connected components $\{A_{i}^{2}\}_{i=1}^{4}$ of $T^{-2m}(CW_{\varepsilon}^{s}(x)) \setminus B_{\delta/4}(y)$ for which $\diam(A_{i}^{2}) = \delta$ for all $1 \le i \le 4$.
Inductively we continue the procedure to generate, for every $r\geq 1$, $2^{r}$ connected components $\{A_{i}^{r}\}_{i=1}^{2^{r}}$ of $T^{-rm}(CW_{\varepsilon}^{s}(x)) \setminus B_{\delta/4}(y)$ such that $\textrm{diam}(A_{i}^{r}) = \delta$ for all $1 \le i \le 2^{r}$. Let $C = \bigcap_{r=1}^{\infty}\bigcup_{i=1}^{2^{r}}T^{mr}(A_{i}^{r})$.
Note that $C \subset CW_{\varepsilon}^{s}(x)$, is compact, and is uncountable (in fact $C$ is a Cantor set). Since $(X,T)$ is expansive, there are at most
countably many $T$-periodic points,
and so we may choose a $T$-aperiodic point $p \in C$. Since $C \subset CW_{\varepsilon}^{s}(x)$, $T^{k}(C) \subset B_{\varepsilon}(T^{k}x)$ for all $k \ge 0$, and it follows that the forward $T$-orbit of $p$ is contained in $B_{\varepsilon}(x)$. Since $B_{\varepsilon}(x) \cap B_{\delta}(y) = \emptyset$, $y$ is not a limit point of the forward orbit of $p$.
We claim that the backward orbit of $p$ does not accumulate on $y$. By construction we have that $T^{-km}(C) \cap B_{\delta/4}(y) = \emptyset$ for all $k \ge 1$.
Continuity implies we may choose $\gamma > 0$ such that if $d(u,v) < \gamma$ then $\max\{d(T^{i}(u),T^{i}(v)) : 1 \le i \le \ell m\} < \delta/4$. Now suppose there exists $k \ge 0$ such that $T^{-k}(p) \in B_{\gamma}(y)$; without loss of generality assume that $k$ is much larger than $\ell m$.
Since $\ell$ and $m$ are relatively prime we may choose $a,b$ such that $k = a\ell+bm$ where $1 \le a \le m$ and $b \ge 1$.
Then $a\ell \le \ell m$ so $T^{-k+a\ell}(p) \in B_{\delta/4}(T^{a\ell}(y)) = B_{\delta/4}(y)$. But $-k+a\ell=-bm$ so this implies $T^{-bm}(p) \in B_{\delta/4}(y)$; since $b \ge 1$, this is a contradiction and the claim is proven.
\end{proof}

\begin{remark}
\label{remark:no-N}
Remarkably, if $(X,T)$ is an $\mathbb{N}$-CAM action such that every non-transitive  point not only has finite orbit but actually is periodic, then $X$ cannot be totally disconnected. This follows from~\cite[Proposition 2.1]{KW} and the discussion that follows. In particular, an $\mathbb{N}$-CAM subshift must have points with finite orbit which are not periodic. \end{remark}

Many existing examples (see Section~\ref{sec:examples}) of CAM systems are {\em algebraic},
meaning that the system consists of the action of a group by automorphisms on some compact group.
However, this can not happen for a $\Z$-action.

        \begin{theorem}\label{thm:noalgebraicZ}
If $\alpha \colon X \to X$ is an automorphism of a compact group $X$, then $(X,\alpha)$ is not a CAM system.
\end{theorem}
To prove this, we use several results in the literature.
Suppose $\beta$ is an action of $\mathbb{Z}^{d}$ by automorphisms on some compact group $X$. Recall that Schmidt~\cite[Section 29]{schmidt} defines an action to be {\em almost minimal} if every closed proper $\beta$-invariant subgroup of $X$ is finite.  Clearly if $\beta$ is CAM, then it is almost minimal.  (We note that we do not know if the converse of this holds.)

\begin{proof}[Proof of Theorem~\ref{thm:noalgebraicZ}]
Without loss of generality, we can assume that $\alpha \colon X \to X$ is transitive, and hence ergodic with respect to the Haar measure on $X$.
If $\alpha$ is CAM, then it is almost minimal.
It is a consequence of~\cite[Theorem 29.2]{schmidt}  that $X$ is either zero-dimensional, a torus, or a solenoid.

First consider the toral case. If $\alpha$ has no eigenvalues of modulus one, then it is hyperbolic and hence not CAM. But if $\alpha$ has an eigenvalue on the unit circle, then it induces a rotation on some subset, and again it is not CAM.  By passing to an inverse limit, the solenoid case follows.

Thus we are left with the case that $X$ is zero dimensional.
We follow ideas in Kitchens~\cite{kitchens}.
Since $X$ is compact and zero dimensional, we can choose an open normal subgroup $K$ containing the identity.
Then the cosets of $K$ are disjoint and partition $X$;
since $X$ is compact, finitely many cosets
$\{a_{i}K\}$ suffice for partitioning $X$ and the subgroup $K$ is closed.

We claim for all $n \in\N$,  there exists $x_{n} \in X$ such that $x_{n}^{-1}\alpha^{n}(x_{n}) \not \in K$. We verify the claim by contradiction, supposing that $K$ does not have this property.  Note that $K$ is infinite (since it is open and $X$ is perfect) and so  there exists $m\in\N$ such that for all $x \in X$ we have $x^{-1}\alpha^{m}(x) \in K$, and hence $\alpha^{m}(x) \in xK$. This implies that $\alpha^{m}(K) \subset K$, and since $K$ is closed, proper and infinite, it follows that $\alpha$ is not CAM, a contradiction.  This verifies the claim about the subgroup $K$.

Since $X/K$ is a finite group,  we can  code orbits using the partition $\{a_{i}K\}$ to get a system $((X/K)^{\mathbb{Z}},\sigma)$ and factor map $\pi \colon X \to X/K$.
Let $Y \subset (X/K)^{\mathbb{Z}}$ be the image of $\pi$. Then $Y$ is a compact totally disconnected topological group and $(Y,\sigma)$ is expansive, and so by Kitchens~\cite[Theorem 1]{kitchens} the system $(Y,\sigma)$ is a one-step shift of finite type. Since $(X,\alpha)$ is transitive, the factor $(Y,\sigma)$ has a point with  dense forward orbit. We claim that $Y$ is infinite.
Suppose instead that $Y$ is finite, and choose $N$ such that $\sigma^{N}=\id$ on $Y$. Then for all $x \in X$, we have $\sigma^{N}(xK) = xK$.  It follows that  $\alpha^{N}(x)K = xK$, and hence $x^{-1}\alpha^{N}(x) \in K$.
However, we chose $K$ such that for all $n \in \mathbb{N}$ there exists $x_{n} \in X$ such that  $x_{n}^{-1}\alpha^{n}(x_{n}) \not \in K$, a contradiction, proving that $Y$ is infinite.
Summarizing, we have that $Y$ is an infinite shift of finite type that is point transitive, and hence contains a proper infinite closed $\sigma$-subset $Z \subset Y$. Let $A = \pi^{-1}(Z)$. Then $A$ is closed, $\alpha$-invariant, infinite, and proper, and so $(X,\alpha)$ is not CAM.
\end{proof}

\subsection{Weakly CAM systems}
\label{sec:weaklyCAM}

\begin{definition}
The topological $G$-system $(X, T)$ is a {\em  weakly CAM system} if all of the following hold:
\begin{enumerate}
\item
The system $(X,T)$ is transitive and the action of $T$ on $X$ is faithful.
\item
The space $X$ contains a dense set of $T$-periodic points.
\item
If $x \in X$, then the orbit of $x$ is either finite or dense in $X$.
\end{enumerate}
\end{definition}

It is immediate that a CAM system is a weakly CAM system (see Proposition~\ref{prop:weak-to-strong}), and to explore the relation between weakly CAM and  CAM we use a result that is well known for subshifts, adapting the argument to the general case.

\begin{lemma}
\label{lemma:expansive}
If $(X,T)$ is infinite and expansive, then it contains an aperiodic point.
\end{lemma}

\begin{proof}
Choose a sequence of distinct elements in $X$
and pass to a convergent subsequence $x_{i} \to x$. If $x$ is aperiodic, we are done and so suppose that $x$ has period $p$.
Since the system is expansive, there are only a finite number of points of period $p$,
and so by removing finitely many terms of the sequence $x_{i}$, we can assume that none of the $x_{i}$ are of period $p$. Let $\varepsilon_{c}$ be an expansivity constant for $(X,T)$. Then for each $i\in\N$, there exists $j_{i}\in\N$ such that $d(T^{j_{i}}(x_{i}),T^{j_{i}+p}(x_{i})) > \varepsilon_{c}$ (if instead for all $k\in\N$ we have that  $d(T^{k}(x_{i}),T^{k+p}(x_{i})) \le \varepsilon_{c}$, then $x_{i} = T^{p}(x_{i})$ contradicting that $x_{i}$ does not have period $p$). Passing to a subsequence if necessary, we can assume that all of the $j_{i}$ are either positive or negative; let us assume they are positive, as the negative case is analogous. Choose for each $i$ the minimal such $j_{i}$ which satisfies the conditions. Then $d(T^{k}(x_{i}),T^{k+p}(x_{i})) \le \varepsilon_{c}$ for all $0 \le k < j_{i}$. Since the sequence $x_{i}$ consists of distinct element and converges to $x$, the $j_{i}$'s must be unbounded. Choosing a limit point $y$ of the sequence $T^{j_{i}}(x_{i})$, it follows that $y$ is backward asymptotic to a point of period $p$. If the point $y$ is periodic, then its period would necessarily be $p$, but by construction, $d(y,T^{p}(y) \ge \varepsilon_{c}$.  Thus the point $y$ is not periodic.
\end{proof}

\begin{proposition}
\label{prop:weak-to-strong}
If $(X,T)$ is CAM, then it is weakly CAM. Moreover, if $(X,T)$ is expansive and weakly CAM, then it is  CAM.
\end{proposition}

\begin{proof}
If $(X,T)$ is CAM and $x \in X$ has infinite orbit, then the orbit closure of $x$ is an infinite closed $T$-invariant subset, and hence is all of $X$.

For the second statement,  suppose $(X, T)$ is expansive and weakly CAM.
Let $Y\subset X$ be an infinite closed $T$-invariant subset. Then $(Y,T|_{Y})$ is an infinite expansive system,
and so by Lemma~\ref{lemma:expansive}, $Y$ contains an aperiodic point. Since $(X,T)$ is weakly CAM, that point has dense orbit under $T$, and hence $Y = X$.
\end{proof}

This proposition motivates the following example showing that the two notions are distinct (a related issue is raised in  Question~\ref{question:expansive}).
\begin{example}{(A weakly CAM but not  CAM system)}
\label{ex:weak-not-strong}
Let $(X,T)$ be a CAM $\mathbb{Z}$-subshift and assume that $(X,T)$ has a fixed point $p$.
(The existence of a $\Z$-CAM system is shown in Section~\ref{sec:Zcamsystem}, and this system has two fixed points.)

We start by defining the space for our new action. Choose a sequence of proper clopen subsets $A_{n}$ in $X$ such that for every $n \in\N$, both of the following are satisfied:
\begin{enumerate}
\item
$A_{n}$ contains $p$.
\item
$A_{n+1}$ is a proper subset of $A_{n}$.
\end{enumerate}
For $n \in \mathbb{N}$, let $\mathcal{U}_{n} = \{\frac{k}{2^{n}} :  0 \le k \le 2^{n}-1\}$ denote the additive group of $2^{n}$-th roots of unity in $S^{1} = \mathbb{R} / \mathbb{Z}$. Let $Z = X \times S^{1}$, and consider the subset $Y$ of $Z$ defined as the set of points $(x,t) \in Z$ such that:

\begin{enumerate}[resume]
\item
If $x \not \in A_1$,  we require that $t=0$.
\item
If $x \in A_1$ and  $n$ is chosen such that $x \in A_{n}$ but $x \not \in A_{n+1}$,  we require $t \in \mathcal{U}_{n+1}$.
\item If $x=p$, there is no constraint on $t$. \end{enumerate}
Note that any point of the form $(p,t)$ is a limit of points of the form $(x,t)$ with  $x \neq p$.

The map $f \colon Y \to Y$ is defined as follows.
\begin{enumerate}[resume]
\item
If $x \not \in A_1$, then $f(x,0) = (T(x),0)$.
\item
\label{item:fixed}
For any $t \in S^{1}$, $f(p,t) = (p,t)$.
\item
Suppose $x \in A_{0}$ and let $n$ be such that $x \in A_{n}$ but $x \not \in A_{n+1}$. Consider $(x,t)$ where $t = \frac{k}{2^{n+1}}  \in \mathcal{U}_{n+1}$. We then define
\[
 f(x,t) =
\begin{cases}
(x,t+\frac{1}{2^{n+1}}) & \text{ for } t \ne \frac{2^{n+1}-1}{2^{n+1}} \\
    (T(x),0) & \text{ for }
    t = \frac{2^{n+1}-1}{2^{n+1}}.
\end{cases}
\]
%
In other words, points of the form $(x,t)$ where $x \in A_{n}$ (and $n$ is maximal such) rotate through the $2^{n+1}$-st roots of unity in the $S^{1}$ term, and at the end of that excursion, map back to $X$ via $T$.
\end{enumerate}

We claim that the map $f \colon Y \to Y$ is a homeomorphism
and the system $(Y,f)$ is weakly CAM but not CAM.

To check that $f$ is a homeomorphism, we consider the various cases. If $(x,t) \in Y$ and $x \ne p$, then continuity of $f$ at $(x,t)$ is immediate.  For a point $(p,t)$ with $t\neq 0$, consider a sequence $(x_{n},t_{n}) \to (p,t)$ as $n\to\infty$. Consider $\{m(n) : x_{n} \in A_{m(n)} \textnormal{ and } m(n) \textnormal{ is maximal among these choices}\}$ (note that by the nesting property, this set is finite).
By construction, $m(n) \to \infty$ as $n \to \infty$, and since $t \ne 0$, the sequence $t_{n}$ is bounded away from 0. Thus $f(x_{n},t_{n}) = (x_{n},t_{n}+ \frac{1}{2^{m(n)}+1}) \to (p,t)$ as $n\to\infty$ and continuity of $f$ at $(p,t)$ follows.   Continuity at $(p,0)$ is similar.

We now consider the system $(Y,f)$.
Since the set $\{(p,t) : t \in S^{1}\}$ is proper, infinite, closed, and $f$-invariant, this system is not CAM.
To check that $(Y,f)$ is weakly CAM, note that if $x$ is periodic then $(x,t)$ is also periodic. Since $X$ is CAM, periodic points are dense in $X$, and so the same  holds for $(Y,f)$. If $x\in X$ is not periodic then since $(X,T)$ is weakly CAM, the orbit of $x$ is dense in $X$.  It follows immediately that the orbit of $(x,0)$ is also  dense in $Y$.
\end{example}

This example is useful in that it shows two things: that a weakly CAM $\mathbb{Z}$-system need not be expansive, and that it need not be totally disconnected. A slight variation of the example can be used to produce a weakly CAM $\mathbb{Z}$-system which is totally disconnected but not expansive.

\section{Groups that admit CAM actions}

\subsection{Residually finite}
\label{sec:residually-finite}
Recall that a group $G$ is {\em residually finite} if for every element $g\in G$ that is not the identity, there is a group homomorphism  from $G$ to a finite group such that $g$ is not mapped to the identity.

\begin{proposition}
\label{prop:residually}
If $G$ is a CAM group, then $G$ is residually finite.
\end{proposition}
\begin{proof}
Assume that the $G$-system $(X, T)$ is CAM and suppose that $g \in G$ is not the identity. Since $G$ acts faithfully and the periodic points are dense, there is a periodic point $p$ such that $T_{g}p \ne p$. It follows that $G$ acts on the finite orbit of $p$ and $g$ acts nontrivially on this orbit.
\end{proof}

We note that this proof only relies on the faithfulness of the action and the existence of a dense set of periodic points, and not the other properties of a CAM group. However, this clarifies that many types of groups, such as  the additive group of rationals  or any infinite simple group, are not CAM. However, we do not know if residual finiteness is the only restriction (see Conjecture~\ref{conj:residual}).

\subsection{Finite index subactions}

The main result in this section is showing that a finite index subgroup of a CAM group is also CAM.  To show this, we start with a lemma on how orbits act under the action of a CAM system.

\begin{lemma}\label{lemma:lemmaforfinindex}
Suppose $(X,T)$ is a $G$-CAM system and $H$ is finite index normal subgroup of $G$. Let $\{g_{1}H,g_{2}H,\ldots,g_{m}H\}$ denote the set of left cosets of $H$ in $G$.  If $Y \subseteq X$ is a compact $H$-invariant subspace containing an infinite $H$-orbit, then $\bigcup_{i=1}^{m}g_{i}(Y) = X$.  Furthermore, the  $H$-periodic points are dense in $Y$.
\end{lemma}
\begin{proof}
We first prove that $\bigcup_{i=1}^{m}g_{i}(Y) = X$. Choose $y \in Y$ which has an infinite $H$-orbit. Then the $G$-orbit of $y$ is infinite, and since the action of $G$ on $X$ is CAM, it follows that the orbit of $y$ is dense in $X$. Let $x \in X$
and choose a sequence  of elements $k_n\in G$ such that $k_n\cdot y \to x$ as $n\to\infty$.
For each $n\in\N$,  there exists $1 \le i(n) \le m$ and
$h_n \in H$ such that $k_n = g_{i(n)}h_n$.  By passing to a subsequence if necessary, we can assume that $i(n) = i$ for all $n\in\N$. It follows that  $g_{i}h_n \cdot y \to x$ as $n\to\infty$, and $h_n \cdot y \to  g_{i}^{-1}(x)$ as  $n\to\infty$.
Since $Y$ is closed and $H$-invariant we have that $g_{i}^{-1}(x) \in Y$, and so $x \in g_{i}(Y)$ for some $i\in\{1, \ldots, m\}$, proving the first statement.

We next show that the boundary $\partial Y$ is a finite set consisting of $G$-periodic points (note that $\partial Y$ is not necessarily $G$-invariant).
To check this, let $B = \bigcup_{i=1}^{m}g_{i}(\partial Y)$.
The set $B$ is closed, and we claim it is $G$-invariant.
Indeed, given $g \in G$, write $g = g_{j}h$ for some $j$ and $h \in H$. Then
\begin{multline*}
g(B) = g\bigl(\bigcup_{i=1}^{m}g_{i}(\partial Y)\bigr) = \bigcup_{i=1}^{m}gg_{i}(\partial Y) =
\bigcup_{i=1}^{m}g_{j}h g_{i}(\partial Y) \\ = \bigcup_{i=1}^{m} g_{j}g_{i} h^{\prime} (\partial Y)
= \bigcup_{i=1}^{m} g_{\ell(i)}h_{\ell(i)}(\partial Y) = \bigcup_{i=1}^{m}g_{i}(\partial Y) = B,
\end{multline*}
where the third to last equality follows from normality of $H$ in $G$ and the second to last from the fact that for each $i$ there exists a unique $\ell(i)$ such that $g_{j}g_{i} = g_{\ell(i)}h_{\ell(i)}$ for some $h_{\ell(i)} \in H$.
Since $Y$ is closed, $\partial Y$ is nowhere dense, and hence so is  $g_{i}(\partial Y)$ for $i=1, \dots, m$.
Thus $B$ is a proper subset of $X$. Therefore,   $B$ is a $G$-invariant subspace of $X$ and since the action of $G$ on $X$ is CAM, it follows that $B$ is finite and only consists of $G$-periodic points. Since $\partial Y \subset B$, it follows that the same holds for $\partial Y$.

Since $G$-periodic points are dense in $X$, it follows that $H$-periodic points are dense in the interior $\textrm{Int}(Y)$ of $Y$. Since $Y = \textrm{Int}(Y) \cup \partial Y$ and $\partial Y$ consists of $H$-periodic points, the second statement follows.
\end{proof}

\begin{theorem}\label{thm:finiteindex}
Any finite index subgroup of a CAM group is  CAM.
\end{theorem}
\begin{proof}
Suppose $(X,G)$ is CAM and let $H$ be a finite index subgroup of $G$.
We first show that we can reduce to the case that $H$ is normal in $G$. Since $H$ has  finite index, there exists a subgroup $K \subseteq H$ such that $K$ is finite index and normal in $G$.
Suppose there exists a compact $K$-invariant subspace $Z \subset X$ such that the system $(Z,K)$ is CAM.
Let $\{\alpha_{i}\}_{i=1}^{m}$ be a set of left coset representatives of $K$ in $H$.
Setting $Z^{\prime} = \bigcup_{i=1}^{m}\alpha_{i}(Z)$, we have that $Z^{\prime}\subseteq X$ is $H$-invariant and compact.
We claim that the system $(Z^{\prime},H)$ is CAM. It is easy to check that periodic points (with respect to the $H$-action) are dense in $Z^{\prime}$, since $K$-periodic points are dense in $Z$ and each $\alpha_{i}$ takes a $K$-periodic point in $Z$ to a $K$-periodic point (and hence $H$-periodic point) in $\alpha_{i}(Z)$.

Now suppose $A$ is a nonempty proper compact $H$-invariant subset of $Z^{\prime}$. It follows that  $A^{\prime} = A \cap \alpha_{j}(Z) \ne \emptyset$ for some $j\in\{1, \dots, m\}$.
Furthermore, $A^{\prime}$ is $K$-invariant, since $K$ is normal in $G$ and both $A$ and $\alpha_{j}(Z)$ are $K$-invariant.
Since $K$ is normal, this  implies that $\alpha_{j}^{-1}(A^{\prime}) \subset Z$ is also $K$-invariant. Since $(Z,K)$ is CAM, we must have $\alpha_{j}^{-1}(A^{\prime}) = Z$, and so $\alpha_{j}(Z) \subset A$. But $A$ is $H$-invariant and so $Z^{\prime} \subset A$, and it follows that the system $(Z', H)$ is CAM.

We can thus assume that $H$  is a finite index,  normal subgroup in $G$.
Let $\mathcal{C}_{H}$ denote the set of all compact $H$-invariant subsets of $X$ which contain some point with an infinite $H$-orbit.
The set $\mathcal{C}_{H}$ is partially ordered by setting $Y_{1} \preceq Y_{2}$ if $Y_{1} \subset Y_{2}$. Suppose $\{Y_{j}\}_{j \in J}$ is a totally ordered collection of such subsets and let $\{g_{1},\ldots,g_{m}\}$ be a set of left coset representatives for $H$ in $G$.

Set $Y_{\infty} = \bigcap_{j \in J}Y_{j}$.  We claim that $Y_{\infty}$ contains a point with infinite $H$-orbit.
Since $H$ has finite index in $G$, it suffices to show that $Y_{\infty}$ contains a point with infinite $G$-orbit.  In fact we show the stronger statement that for any $x \in X$  with infinite $G$-orbit, we have that $Y_{\infty}$  contains at least one point from the set $\{g_{1}^{-1}x,\ldots,g_{m}^{-1}x\}$; since each of these points has infinite $G$-orbit,  the claim follows.
Arguing by contradiction, if this does not hold, then for each $1 \le i \le m$, there exists $j_{i}$ such that $g_{i}^{-1}(x) \not \in Y_{j_{i}}$.
We set $J = \max_{1 \le i \le m}\{j_{i}\}$, and note that $J$ is finite since any finite subset of a totally ordered set has a maximum.
Then $g_{i}^{-1}(x) \not \in Y_{J}$ for all $i\in\{1, \dots, m\}$, since $Y_{J} \subset Y_{j_{i}}$ for all $i\in\{1, \dots, m\}$. This then implies  that $x \not \in g_{i}Y_{J}$ for all $i\in\{1, \dots, m\}$.
However, $Y_{J}$ belongs to $\mathcal{C}_{H}$ and hence contains some point with  infinite $H$-orbit.  Thus by Lemma~\ref{lemma:lemmaforfinindex}, it follows that  $\bigcup_{i=1}^{m}g_{i}(Y_{j}) = X$, a contradiction.
In particular, $Y_{\infty}$ contains a point with infinite $H$-orbit.

Thus we have that $Y_{\infty}$ belongs to $\mathcal{C}_{H}$ and is an upper bound for the collection $\{Y_{j}\}_{j \in J}$. By Zorn's Lemma, there is a maximal element in $Z \in \mathcal{C}_{H}$, and this element $Z$ is a compact $H$-invariant subset  of $X$ containing some point with an infinite $H$-orbit.
We claim the action of $H$ on $Z$ is a CAM $H$-system. If $x \in Z$ has infinite $H$-orbit, then that orbit is dense, for otherwise its orbit closure would be a proper compact $H$-invariant subspace of $Z$ containing an infinite $H$-orbit, contradicting the maximality of $Z$.
Thus if $A$ is a proper compact $H$-invariant subset of $Z$, then $A$  consists only of $H$-periodic points. Consider $A^{\prime} = \bigcup_{i=1}^{m}\alpha_{i}(A) \subset X$. Then $A^{\prime}$ is compact, nonempty $G$-invariant set in $X$, and it is proper since it consists only of $G$-periodic points. Since $(X,G)$ is CAM it follows that $A^{\prime}$ is finite and so $A$ is also finite.

Finally, by Lemma~\ref{lemma:lemmaforfinindex}, the $H$-periodic points are dense in $Z$.
\end{proof}

We note the following corollary.
\begin{corollary}
\label{cor:extend-action}
Suppose $H$ is group acting by a CAM action on $X$ and that $H$ is a finite index subgroup of a group $G$. If there is an action of $G$ on $X$ extending the action of $H$, then the $G$ action is also CAM.
\end{corollary}

\begin{example}[Free group]
\label{example:free}
Since ${\rm SL}_{2}(\mathbb{Z})$ is CAM (see Example~\ref{example:SLdZ}), Theorem~\ref{thm:finiteindex} implies the free group on two generators is CAM.
 More generally, the free group on $n$ generators for any $n\geq 3$ is a finite index subgroup of the free group on $2$ generators, and so is also CAM.
\end{example}
Another corollary of  Theorem~\ref{thm:finiteindex} gives a sharpening of Corollary~\ref{cor:extend-action}. \begin{corollary}\label{cor:finindex}
Suppose $G$ acts on $X$ by a CAM action and let $H \subset G$ be a finite index normal subgroup. Suppose there exists $x \in X$ whose $H$-orbit is dense in $X$. Then the action of $H$ on $X$ is CAM.
\end{corollary}
\begin{proof}
Let $Z$ denote an $H$-invariant set in $X$ on which the action of $H$ is CAM, which exists by Theorem~\ref{thm:finiteindex}. By Lemma~\ref{lemma:lemmaforfinindex} we have $\bigcup_{i=1}^{m}g_{i}(Z) = X$. Moreover, it is straightforward to check that each $g_{i}(Z)$ is $H$-invariant and the action of $H$ on any $g_{i}(Z)$ is also CAM. Now if $x \in X$ has dense $H$-orbit we may choose $g_{i}$ such that $x \in g_{i}(Z)$, and then $X = g_{i}(Z)$.
\end{proof}
It follows immediately from Theorem~\ref{thm:finiteindex} and Example~\ref{example:autos} that any finite index subgroup of the automorphism group of a mixing shift of finite type is CAM.  Combining the same example with Corollary~\ref{cor:finindex} we obtain the stronger statement for shifts of finite type.
\begin{corollary}
If  $H$ is a finite index subgroup of the automorphism group of a mixing shift of finite type $(X, \sigma)$, then the $H$-subaction on $X$ is a CAM system.
%
%
\end{corollary}


\section{{Existence of a {$\Z$}-CAM system}}
\label{sec:Zcamsystem}

\subsection{Construction of a CAM $\mathbb{Z}$-system with multiple nonatomic, ergodic invariant measures}

In this section we show that a $\Z$-CAM subshift exists. Moreover, the $\Z$-CAM subshift that we construct has two different, non-atomic ergodic probability measures.

\label{sec:construction}

\subsubsection{Setting notation}
We build the system as a subshift on the alphabet $\CA = \{0,1\}$.
The system is defined recursively, defining the language of the system in levels.
The first three levels play a special role in starting the inductive process of defining the words on higher levels.
Beginning with level four, we use a recursive formula to define the words, and the words constructed are of two types: those used to form the periodic points, denoted by the letter $w$ with appropriate subscripts, and those designed to produce dense orbits, denoted by the letters $a$ and $b$ with appropriate subscripts.
Throughout the construction, the periodic words have two subscripts, with the second coordinate denoting the level at which
the word is introduced, whereas the density words come in two types ($a$ and $b$) but only have a single subscript, also indicating the level at which such a word is introduced.
We remark that, in our construction, we build a single infinite word that has the desired property, and its orbit closure is the CAM system.  We do not explicitly build the language in our construction.

We fix a sequence $\{\varepsilon_k\}_{k=1}^{\infty}$ for use throughout our construction, which we refer to as a
{\em frequency sequence}: assume that $\{\varepsilon_k\}_{k=1}^{\infty}$ is a sequence of non-negative real numbers that satisfy $\sum_{k=1}^{\infty}\varepsilon_k<1/3$, and further assume that for any $N\geq1$,  we have
\begin{equation}\label{eq:sum}
\sum_{n=N}^{\infty}\varepsilon_n<\frac{1}{3^N}.
\end{equation}
To keep track of frequency counts,
for words $u,v\in\{0,1\}^*$ of finite length, define $\mathcal{N}(u,v)$ to be the number of times $u$ occurs as a subword of $v$.

\subsubsection{Level 1 words (introduction of the words $w$)}
There are $2$ words on this level,
$$
w_{(1,1)}=0 \quad \text{ and } \quad w_{(2,1)}=1,
$$
and we call these the {\em level-1 words}.  Set $\mathcal{W}_1=\{w_{(1,1)},w_{(2,1)}\}$.
\subsubsection{Level 2 words (introduction of the words $a$ and $b$)} There are $4$ words on this level,
playing two distinct roles as distinguished by two types: we have periodic words $w$ and density words $a,b$.  This level also introduces a parameter $n_2>1$ that controls the densities (note that this is the first level with such a parameter and so there is no $n_1$; the meaning of this is  clarified momentarily). We define:
\begin{eqnarray*}
w_{(1,2)}&=&w_{(1,1)}^{n_2+1} \\ \\
w_{(2,2)}&=&w_{(2,1)}^{n_2+1} \\ \\
a_2&=&w_{(1,1)}w_{(2,1)}^{n_2} \\ \\
b_2&=&w_{(1,1)}^{n_2}w_{(2,1)}
\end{eqnarray*}
The parameter $n_2$ is chosen sufficiently large such that $\mathcal{N}(0,a_2)/|a_2|<\varepsilon_1$ and $\mathcal{N}(1,b_2)/|b_2|<\varepsilon_1$.  In other words, $a_2$ consists almost entirely  of $1$'s, $b_2$  almost entirely  of $0$'s, and all of the words constructed on this level have the same length.  We call the four words introduced at this point the {\em level-2 words} and we write $\mathcal{W}_2=\{w_{(1,2)},w_{(2,2)},a_2,b_2\}$.
For use in verifying properties of the words constructed at higher levels, we describe the basic properties of the words constructed at level 2.
\begin{proposition}\label{prop:level1}
For all distinct words $u,v\in\mathcal{W}_2$, the word $u$ does not occur as a subword of $vv$.
\end{proposition}
\begin{proof}
First we establish the statement when $v\in\{w_{(1,2)},w_{(2,2)}\}$ and $u\in\mathcal{W}_2$ is distinct from $v$.  Neither $a_2$ nor $b_2$ occurs as a subword of $w_{(i,2)}w_{(i,2)}$ for $i=1,2$, because those words are $1$-periodic whereas $a_2$ and $b_2$ are not.  The word $w_{(1,2)}$ does not occur as a subword of $w_{(2,2)}w_{(2,2)}$, since the symbol $0$ does not occur in $w_{(2,2)}$ but does occur in $w_{(1,2)}$ and similarly $w_{(2,2)}$ does  not occur as a subword of $w_{(1,2)}w_{(1,2)}$.  This establishes the statement when $v\in\{w_{(1,2)},w_{(2,2)}\}$.

Next we establish the statement when $v\in\{a_2,b_2\}$ and $u\in\mathcal{W}_2$ is distinct from $v$.  Neither $w_{(1,2)}$ nor $w_{(2,2)}$  occurs as a subword of $a_2a_2$, as the longest $1$-periodic subword of $a_2a_2$ has length $n_2\cdot|w_{(2,1)}|<|w_{(i,2)}|$ for $i=1,2$.  Similarly, neither can occur as a subword of $b_2b_2$.  We claim that the word $a_2$ does not occur as a subword of $b_2b_2$ (the argument that $b_2$ does not occur as a subword of $a_2a_2$ is analogous).  By construction, $\mathcal{N}(0,a_2)/|a_2|<\varepsilon_1$ and so $\mathcal{N}(0,a_2)<\varepsilon_1|a_2|$.  Therefore $\mathcal{N}(1,a_2)\geq(1-\varepsilon_1)|a_2|$.  Similarly, $\mathcal{N}(0,b_2b_2)\geq(1-\varepsilon_1)|b_2b_2|$.  If $a_2$ occurs as a subword of $b_2b_2$ then, using the fact that $|a_2|=\frac{1}{2}|b_2b_2|$, we obtain
$$
\mathcal{N}(1,b_2b_2)\geq\mathcal{N}(1,a_2)\geq(1-\varepsilon_1)|a_2|=\frac{1-\varepsilon_1}{2}|b_2b_2|.
$$
This implies that $\mathcal{N}(0,b_2b_2)\leq\left(1-\frac{1-\varepsilon_1}{2}\right)|b_2b_2|=\frac{1+\varepsilon_1}{2}|b_2b_2|$.  Combined with our earlier observation that $\mathcal{N}(0,b_2b_2)\geq(1-\varepsilon_1)|b_2b_2|$, we get $1-\varepsilon_1\leq\frac{1+\varepsilon_1}{2}$.  However, this can not occur because
we assumed that  $\varepsilon_1<\sum_{k=1}^{\infty}\varepsilon_k<1/3$.  Therefore the statement holds when $v\in\{a_2,b_2\}$.
\end{proof}

\subsubsection{Level $3$ words (recursive definitions of periodic and density words)}
There are $6$ words on this level,  again playing distinct roles distinguished by two types: periodic words $w$ and density words $a,b$.  Again we introduce a parameter $n_3 >1 $ to control the densities.   Define the words by setting
\begin{eqnarray*}
w_{(1,3)}&=&w_{(1,2)}^{5n_3+4} \\ \\
w_{(2,3)}&=&w_{(2,2)}^{5n_3+4} \\ \\
w_{(3,3)}&=&a_2^{5n_3+4 }\\ \\
w_{(4,3)}&=&b_2^{5n_3+4} \\ \\
a_3&=&a_2^{n_3}w_{(1,2)}a_2^{n_3}w_{(2,2)}a_2^{n_3}a_2a_2^{n_3}b_2a_2^{n_3} \\ \\
b_3&=&b_2^{n_3}w_{(1,2)}b_2^{n_3}w_{(2,2)}b_2^{n_3}a_2b_2^{n_3}b_2b_2^{n_3}
\end{eqnarray*}
and choose $n_3$ such  that $\mathcal{N}(0,a_3)/|a_3|<\varepsilon_1+\varepsilon_2$ and $\mathcal{N}(1,b_3)/|b_3|<\varepsilon_1+\varepsilon_2$ (we note that this is possible because $\mathcal{N}(0,a_2)/|a_2|<\varepsilon_1$ and $a_3$ is mainly made of concatenated copies of $a_2$, similarly $\mathcal{N}(1,b_2)/|b_2|<\varepsilon_1$).
We further assume that $n_3$ has been chosen sufficiently large such that for $u\in\{w_{(1,2)},w_{(2,2)},b_2\}$,  we have
\begin{equation}\label{eq:fraction1}
\frac{\mathcal{N}(u,a_3a_3)}{|a_3a_3|}<\frac{\varepsilon_2}{|u|(2|u|-1)},
\end{equation}
and note that this can be done since $a_3a_3$ is mainly made of concatenated copies of $a_2a_2$ and  $u$ does not arise as a subword of $a_2a_2$. Note that if $|u|=1$ then the denominator on the right-hand side of this inequality is $1$, so this formula generalizes the analogous inequality (and assumptions about the size of $n_2$) in the previous level.  Using the same reasoning, this choice of $n_3$ ensures that if $u\in\{w_{(1,2)},w_{(2,2)},a_2\}$, then we have
\begin{equation}\label{eq:fraction2}
\frac{\mathcal{N}(u,b_3b_3)}{|b_3b_3|}<\frac{\varepsilon_2}{|u|(2|u|-1)}.
\end{equation}
Finally note that by choosing $n_3$ sufficiently large, we can guarantee that $|a_2|/|a_3|$ is as small as desired.  For use in the proof of the next proposition describing occurrences of words and subwords, we have a specific constant we would like it to be smaller than.  Let $p$ be the minimal period of the word $a_2a_2$ which, by symmetry of the construction, is also the minimal period of $b_2b_2$.  We assume that $n_3$ is  chosen to be sufficiently large such that
\begin{equation}\label{eq:fraction3}
\frac{|a_2|}{|a_3|}<\frac{1}{6p}-\frac{1}{9|a_2|}.
\end{equation}
We call the words constructed on this level the {\em level-3 words} and set $\mathcal{W}_3$ to be the set of all level-3 words.
\begin{proposition}\label{prop:level2}
For any choice of  distinct words $u,v\in\mathcal{W}_3$, the word $u$ does not occur as a subword of $vv$.
\end{proposition}
\begin{proof}
We check that for each $v\in\mathcal{W}_3$ and  $u\in\mathcal{W}_3\setminus\{v\}$, the word $u$ does not occur as a subword of $vv$, checking cases depending on the type of word $v$.

First consider when  $v=w_{(i,3)}$ for some $i\in\{1,2\}$.   We consider two cases, depending on the choice of $u$. First suppose that $u\in\{w_{(3,3)},w_{(4,3)},a_3,b_3\}$.  By Proposition~\ref{prop:level1}, neither $a_2$ nor $b_2$ occurs as a subword of $w_{(i,2)}w_{(i,2)}$.  The word $vv=w_{(i,3)}w_{(i,3)}$ is the self-concatenation of a large number of copies of $w_{(i,2)}$ and so neither $a_2$ nor $b_2$ occurs as a subword of $vv$.  But at least one of $a_2$ and $b_2$ occurs as a subword of $u$, and so $u$ cannot be a subword of $vv$.
Next suppose $u=w_{(j,3)}$ where $j\in \{1,2\}$ and $j\neq i$.  The word $u=w_{(j,3)}$ is not a subword of $vv=w_{(i,3)}w_{(i,3)}$ because $w_{(j,2)}$ is not a subword of $w_{(i,2)}w_{(i,2)}$ and, as noted, $vv$ is the self-concatenation of copies of $w_{(i,2)}$ while $w_{(j,3)}$ is the self-concatenation of copies of $w_{(j,2)}$.  Thus
$u$ does not occur as a subword of $vv$ for any $u\in\mathcal{W}_3\setminus\{v\}$.  Thus the statement holds when $v\in\{w_{(1,3)},w_{(2,3)}\}$.

Next consider when $v=w_{(i,3)}$ for some $i\in\{3,4\}$.  Again we have two cases, depending on the choice of $u$.
First suppose $u\in\{w_{(1,3)},w_{(2,3)},a_3,b_3\}$.
By Proposition~\ref{prop:level1}, neither $w_{(1,2)}$ nor $w_{(2,2)}$ occurs as a subword of $a_2a_2$ or of $b_2b_2$.  The word $vv=w_{(i,3)}w_{(i,3)}$ is the self-concatenation of a large number of copies of $a_2$ or of $b_2$, and so neither $w_{(1,2)}$ nor $w_{(2,2)}$ occurs as a subword of $vv$.
However at least one of $w_{(1,2)}$ and $w_{(2,2)}$ occurs in $u$, and so $u$ is not a subword of $vv$.  We next consider when $j\in\{3,4\}\setminus\{i\}$ and let $u=w_{(j,3)}$.  Then one of the words $u$ and $v$ is the self-concatenation of many copies of $a_2$ and the other is the self-concatenation of many copies of $b_2$.  By Proposition~\ref{prop:level1}, $a_2$ is not a subword of $b_2b_2$ and $b_2$ is not a subword of $a_2a_2$, and so $u$ is not a subword of $vv$.  This establishes the statement when $v\in\{w_{(3,3)},w_{(4,3)}\}$.

Finally consider when $v\in\{a_3,b_3\}$.  We give the argument when $v=a_3$, and the argument for $b_3$ is similar. Again, we have two cases, depending on the choice of $u$.
First suppose $u\in\{w_{(1,3)},w_{(2,3)},w_{(4,3)},b_3\}$.  By~\eqref{eq:fraction1},
we know that $\mathcal{N}(x,a_3a_3)/|a_3a_3|<\varepsilon_2/|x|(2|x|-1)$ for $x\in\mathcal{W}_2\setminus\{a_2\}$.  For any such $x$ and any particular occurrence of $x$ in $a_3a_3$, there are $2|x|-1$ subwords of $a_3a_3$ that have length $|x|$ and partially (or completely) overlap this occurrence of $x$.  This means that, for any $x\in\mathcal{W}_2\setminus\{a_2\}$, if we look at the collection of all locations within $a_3a_3$ where $x$ occurs, we have
$$
\frac{\text{subwords of length $|x|$ in $a_3a_3$ that partially overlap an occurrence of $x$}}{|a_3a_3|}<\frac{\varepsilon_2}{|x|}.
$$
Let $p$ be the minimal period of the word $a_2a_2$.  Note that $p\leq|a_2|$.  Since all four words in $\mathcal{W}_2$ have the same length, we deduce that
\begin{equation}\label{eq:calc}
\frac{\mathcal{N}(a_2,a_3a_3)}{|a_3a_3|}\geq\frac{1}{p}-3\cdot\frac{\varepsilon_2}{|a_2|}
\end{equation}
because $a_3a_3$ is made by concatenating words in $\mathcal{W}_2$.  On the other hand, we claim that $\mathcal{N}(a_2,u)\leq6|a_2|$. To check this, note that if $u=b_3$, any occurrence of $a_2$ in $u$ must partially overlap $w_{(1,2)}$, $w_{(2,2)}$ or $a_2$ in the definition of $b_3$, because it does not occur as a subword of $b_2b_2$, and there are only $6|a_2|$ locations that have such overlaps.  If $u\neq b_3$ then $a_2$ does not occur in $u$, by Proposition~\ref{prop:level1} and the claim follows.
If $u$ occurred as a subword of $a_3a_3$, such an occurrence would account for half of the letters in $a_3a_3$ and we could write
$$
a_3a_3=xuy
$$
where $|x|+|y|=\frac{|a_3a_3|}{2}$.  We refer to this decomposition of $a_3a_3$ in what follows.  Any occurrence of $a_2$ in $a_3a_3$ is either entirely within $u$, entirely within $x$, entirely within $y$, or partially overlaps $x$ (or $y$) and partially overlaps $u$.  There are at most $|x|+|y|=|a_3a_3|/2$ many possible locations for $a_2$ to occur in $a_3a_3$ that are not entirely within $u$, and these locations consist of two disjoint intervals of starting points.  Because the minimal period of $a_2a_2$ is $p$, within either of these intervals, the density with which $a_2$ occurs is at most $1/p$ as any higher density would force two occurrences of $a_2$ whose starting points differ by less than $p$ and therefore contradicting minimality of $p$.  This implies that
\begin{eqnarray*}
\frac{\mathcal{N}(a_2,a_3a_3)}{|a_3a_3|}&\leq&\frac{\text{$\#$ times $a_2$ occurs not entirely in $u$}}{|a_3a_3|}+\frac{\text{$\#$ times $a_2$ occurs in $u$}}{|a_3a_3|} \\
&\leq&\frac{1}{2p}+\frac{6|a_2|}{|a_3a_3|}=\frac{1}{2p}+3\cdot\frac{|a_2|}{|a_3|}.
\end{eqnarray*}
But, by construction, by Equation~\eqref{eq:fraction3} we have that $|a_2|/|a_3|<\frac{1}{6p}-\frac{1}{9|a_2|}$, and so $\mathcal{N}(a_2,a_3a_3)/|a_3a_3|<\frac{1}{p}-\frac{1}{3|a_2|}$.  By our choice of frequency sequence in~\eqref{eq:sum}, it follows that  $\varepsilon_2<\sum_{k=2}^{\infty}\varepsilon_k<1/9$ and so equation~\eqref{eq:calc} implies that $\mathcal{N}(a_2,a_3a_3)/|a_3a_3|\geq\frac{1}{p}-\frac{1}{3|a_2|}$.  This is a contradiction, and so $u$ cannot occur as a subword of $vv=a_3a_3$.
Lastly, suppose that $u=w_{(3,3)}$.
By construction, $u$ is the self-concatenation of $5n_3+4$ copies of $a_2$.  Then $u$ cannot occur as a subword of $vv=a_3a_3$ because $$
a_3a_3=a_2^{n_3}w_{(1,2)}a_2^{n_3}w_{(2,2)}a_2^{n_3}a_2a_2^{n_3}b_2a_2^{n_3}a_2^{n_3}w_{(1,2)}a_2^{n_3}w_{(2,2)}a_2^{n_3}a_2a_2^{n_3}b_2a_2^{n_3}
$$
and so the largest power $m$ such that $a_2^m$ occurs as a subword of $a_3a_3$ satisfies $m\leq2n_3+2$, by Proposition~\eqref{prop:level1} (any larger power would force an occurrence of an element of $\mathcal{W}_2\setminus\{a_2\}$  as a subword of $a_2^m$).
This establishes the statement when $v=a_3$.  The argument for $v=b_3$ is similar, with the roles played by $a_2$ and $b_2$ switched.
\end{proof}

\subsubsection{Preparing to build the level $k+1$ words (inductive assumptions)}
For $k\geq3$,  we construct the words on level $k+1$ using the words on level $k$.  Again they come in two varieties (periodic words and density words), and again we have a parameter $n_{k+1}> 1$ chosen to guarantee that the words have the desired properties.  Inductively, we assume that we have constructed the level-$k$ words: these are words $w_{(i,k)}$ for all $1\leq i\leq2k-2$, words $a_k$ and $b_k$, with all of these words having equal length, and we denote the collection of all level-$k$ words by $\mathcal W_k$.
We further assume that we have defined the parameter $n_k\in\N$, and these constructions satisfy the following properties:
	\begin{enumerate}[label=(\roman*)]
	\item \label{item:firstk}
 For $1\leq i\leq2k-4$, we have
	$$
	w_{(i,k)}=w_{(i,k-1)}^{(2k-1)n_k+2(k-1)}.
	$$
	\item We have
	$$
	w_{(2k-3,k)}=a_{k-1}^{(2k-1)n_k+2(k-1)}.
	$$
	\item We have
	$$
	w_{(2k-2,k)}=b_{k-1}^{(2k-1)n_k+2(k-1)}.
	$$
	\item We have
	$$
	a_k=\left(\prod_{i=1}^{2k-4}a_{k-1}^{n_k}w_{(i,k-1)}\right)\cdot a_{k-1}^{n_k}\cdot\left(a_{k-1}\right)\cdot a_{k-1}^{n_k}\cdot\left(b_{k-1}\right)\cdot a_{k-1}^{n_k}.
	$$
	\item
 \label{item:lastk}
 We have
	$$
	b_k=\left(\prod_{i=1}^{2k-4}b_{k-1}^{n_k}w_{(i,k-1)}\right)\cdot b_{k-1}^{n_k}\cdot\left(a_{k-1}\right)\cdot b_{k-1}^{n_k}\cdot\left(b_{k-1}\right)\cdot b_{k-1}^{n_k}.
	$$
	\end{enumerate}
We further assume that for any word $u$ constructed on some level $m$ with $m<k$, other than when $u=a_m$, we have
\begin{equation}\label{eq:fraction4}
\frac{\mathcal{N}(u,a_ka_k)}{|a_ka_k|}<\frac{1}{|u|(2|u|-1)}\sum_{j=m}^{k-1}\varepsilon_j
\end{equation}
and that for any $m<k$ and any level-$m$ word $u$ other than $u=b_m$, we have
\begin{equation*}
\frac{\mathcal{N}(u,b_kb_k)}{|b_kb_k|}<\frac{1}{|u|(2|u|-1)}\sum_{j=m}^{k-1}\varepsilon_j.
\end{equation*}
Finally we assume that
\begin{equation}\label{claim:new}
\text{for any $u,v\in\mathcal{W}_k$ with $u\neq v$, the word $u$ does not occur as a subword of $vv$.}
\end{equation}

\subsubsection{Level-$(k+1)$ words (recursive definition)}
We apply the analogous procedure used for  level $3$  and  construct the periodic and density words, this time assuming the properties given in~\eqref{item:firstk} through~\eqref{item:lastk}.
	\begin{enumerate}[label=(\roman*), resume]
	\item \label{item:firstk+1}
 For  $1\leq i\leq2k-2$, define the first group of periodic words  by setting
	$$
	w_{(i,k+1)}=w_{(i,k)}^{(2k+1)n_{k+1}+2k}.
	$$
	\item Define the next periodic word by setting
	$$
	w_{(2k-1,k+1)}=a_k^{(2k+1)n_{k+1}+2k}.
	$$
	\item Define the last periodic word by setting
	$$
	w_{(2k,k+1)}=b_k^{(2k+1)n_{k+1}+2k}.
	$$
	\item Define the first type of density word by setting
	$$
	a_{k+1}=\left(\prod_{i=1}^{2(k+1)-4}a_k^{n_{k+1}}w_{(i,k)}\right)\cdot a_k^{n_{k+1}}\cdot\left(a_k\right)\cdot a_k^{n_{k+1}}\cdot\left(b_k\right)\cdot a_k^{n_{k+1}}.
	$$
	\item
 \label{item:lastk+1}
Define the second type of density word by setting
	$$
	b_{k+1}=\left(\prod_{i=1}^{2(k+1)-4}b_k^{n_{k+1}}w_{(i,k)}\right)\cdot b_{k}^{n_{k+1}}\cdot\left(a_{k}\right)\cdot b_{k}^{n_{k+1}}\cdot\left(b_{k}\right)\cdot b_{k}^{n_{k+1}}.
	$$
	\end{enumerate}
 We choose the parameter $n_{k+1}$ sufficiently large such that for every $m<k+1$ and every word $u\in\mathcal{W}_m\setminus\{a_m\}$, we have
\begin{equation}\label{eq:fraction5}
\frac{\mathcal{N}(u,a_{k+1}a_{k+1})}{|a_{k+1}a_{k+1}|}<\frac{1}{|u|(2|u|-1)}\sum_{j=m}^k\varepsilon_j.
\end{equation}
When $m<k$, recall that $\mathcal{N}(u,a_ka_k)/|a_ka_k|<(\sum_{j=m}^{k-1}\varepsilon_j)/|u|(2|u|-1)$ by~\eqref{eq:fraction4}.
Therefore it is possible to satisfy~\eqref{eq:fraction5} because, for $n_{k+1}$ large, the word $a_{k+1}$  primarily consists of concatenated copies of $a_k$ and so we can make the left hand side of~\eqref{eq:fraction5} as close as we want to $(\sum_{j=m}^{k-1}\varepsilon_j)/|u|(2|u|-1)$; in particular we can make it less than the right hand side of~\eqref{eq:fraction5}.  When $m=k$ it is possible to satisfy~\eqref{eq:fraction5} because $u$ does not occur as a subword of $a_ka_k$ and, again, $a_{k+1}$ is primarily made of concatenated copies of $a_k$.  Similarly we choose $n_{k+1}$  sufficiently large  such that if $u\in\mathcal{W}_m\setminus\{b_m\}$, then we have
\begin{equation}\label{eq:fraction6}
\frac{\mathcal{N}(u,b_{k+1}b_{k+1})}{|b_{k+1}b_{k+1}|}<\frac{1}{|u|(2|u|-1)}\sum_{j=m}^k\varepsilon_j.
\end{equation}
For our next estimate, let $p_k$ be the minimal period of the word $a_ka_k$.  For use in a future argument, we finally require that $n_{k+1}$ is sufficiently large such that
\begin{equation}\label{eq:inductive-bound}
\frac{|a_k|}{|a_{k+1}|}<\frac{1}{(4k-2)p_k}-\frac{1}{3^k\cdot|a_k|}.
\end{equation}
This is possible as long as the right-hand side of the inequality is positive.  Note that $p_k\leq|a_k|$ and that $4k-2<3^k$ because $k\geq3$, and that all of the words constructed on level $k+1$ have equal length.
We call the words constructed in~\eqref{item:firstk+1} through~\eqref{item:lastk+1} the {\em level-$(k+1)$ words} and let $\mathcal{W}_{k+1}$ denote the set of all level-$(k+1)$ words. Therefore, the only thing that is required now to complete our inductive construction is to prove the following proposition.

\begin{proposition}\label{prop:levelkplus1}
For any choice of distinct words  $u,v\in\mathcal{W}_{k+1}$, the word $u$ does not occur as a subword of $vv$.
\end{proposition}
\begin{proof}
As on level three, fixing some  $v\in\mathcal{W}_{k+1}$,  we argue that for all $u\in\mathcal{W}_{k+1}\setminus\{v\}$, the statement holds.  Again, we split the argument into cases depending on the type of the word $v$.

First suppose that $v\in\{w_{(i,k+1)}\colon1\leq i\leq2k-2\}$ for some $1\leq i\leq2k-2$.  We start by considering the case that $u\in\{w_{(2k-1,k+1)},w_{(2k,k+1)},a_{k+1},b_{k+1}\}$.  By~\eqref{claim:new}, neither $a_k$ nor $b_k$ occurs as a subword of $w_{(i,k)}w_{(i,k)}$,
and by construction, at least one of $a_k$ or $b_k$ occurs as a subword of $u$.  Since $vv=w_{(i,k+1)}w_{(i,k+1)}$ is the self-concatenation of a large number of copies of $w_{(i,k)}$, $u$ cannot be a subword of $vv$.  Next consider when $u=w_{(j,k+1)}$ for some
$j\in\{1,2,\dots,2k-2\}$ with $j\neq i$.  Again by~\eqref{claim:new}, $w_{(i,k)}$ is not a subword of $w_{(j,k)}w_{(j,k)}$.  But $w_{(i,k+1)}$ is the concatenation of many copies of $w_{(i,k)}$ and $w_{(j,k+1)}$ is the concatenation of many copies of $w_{(j,k)}$, and so an occurrence of $u$ in $vv$ forces an occurrence of $w_{(i,k)}$ in the word $w_{(j,k)}w_{(j,k)}$, a contradiction.  Thus $u$ cannot occur as a subword of $vv$.  This establishes the statement for $v\in\{w_{(i,k+1)}\colon1\leq i\leq2k-2\}$.

Next assume $v=w_{(i,k+1)}$ for some $i\in\{2k-1,2k\}$.  We first consider when $u\in\{w_{(j,k+1)}\colon1\leq j\leq2k-2\}\cup\{a_{k+1},b_{k+1}\}$.  By~\eqref{claim:new}, none of the words in the set $\{w_{(j,k)}\colon1\leq j\leq2k-2\}$ occurs as a subword of $a_ka_k$ or of $b_kb_k$.  But by construction, the word $w_{(i,k+1)}w_{(i,k+1)}$ is the self-concatenation of a large number of copies of $a_k$ or $b_k$, and  we know that at least one of the words in $\{w_{(j,k)}\colon1\leq j\leq2k-2\}$ occurs as a subword of $u$.  Therefore $u$ cannot occur as a subword of $vv$.  Now instead assume $j\in\{2k-1,2k\}$ with $j\neq i$.  This time, one of the words $w_{(i,k+1)}$ and $w_{(j,k+1)}$ is the concatenation of many copies of $a_k$ and the other is the concatenation of many copies of $b_k$.  By~\eqref{claim:new}, $a_k$ is not a subword of $b_kb_k$ and $b_k$ is not a subword of $a_ka_k$.  Thus $u$ cannot be a subword of $vv$, proving the statement for $v\in\{w_{(2k-1,k+1)},w_{(2k-2,k+1)}\}$.

We are left with checking the case that
$v\in\{a_{k+1},b_{k+1}\}$.
As on level three, we only include the argument when  $v=a_{k+1}$, as the other case is similar.  As before, there are two possibilities and we treat them separately.  First suppose $u\in\{w_{(j,k+1)}\colon1\leq j\leq2k-2\}\cup\{w_{(2k,k+1)},b_{k+1}\}$.  By~\eqref{eq:fraction5}, for any $w\in\mathcal{W}_k\setminus\{a_k\}$ we have that
$$
\frac{\mathcal{N}(w,a_{k+1}a_{k+1})}{|a_{k+1}a_{k+1}|}<\frac{\varepsilon_k}{|w|(2|w|-1)}.
$$
For each particular occurrence of $w$ in $a_{k+1}a_{k+1}$, there are $2|w|-1$ subwords of length $|w|$ that partially overlap this occurrence.  This means that
$$
\frac{\text{\# of length $|w|$ subwords of $a_{k+1}a_{k+1}$ that partially overlap an occurrence of $w$}}{|a_{k+1}a_{k+1}|}<\frac{\varepsilon_k}{|w|}.
$$
Recall that $p_k$ is the minimal period of the word $a_ka_k$.  Note that $p_k\leq|a_k|$.  Since $|\mathcal{W}_k\setminus\{a_k\}|=2k-1$ and all words in $\mathcal{W}_k$ have the same length, we deduce from~\eqref{eq:sum} that
\begin{equation}\label{eq:estimate-k}
\frac{\mathcal{N}(a_k,a_{k+1}a_{k+1})}{|a_{k+1}a_{k+1}|}\geq\frac{1}{p_k}-\frac{(2k-1)\varepsilon_k}{|a_k|}.
\end{equation}
On the other hand, by~\eqref{claim:new}, $a_k$ does not occur as a subword of $w_{(i,k+1)}$ for any $k\leq2k-2$ (because it is not a subword of $w_{(i,k)}w_{(i,k)}$), $a_k$ does not occur as a subword of $w_{(2k,k+1)}$ (because it is not a subword of $b_kb_k$), and $a_k$ occurs as a subword of $b_{k+1}$ at most $(4k-2)\cdot|a_k|$ times (because it does not occur as a subword of $b_kb_k$
and so all occurrences have to partially overlap $b_k$ and some $w_{(i,k)}$ or partially overlap the copy of $a_k$ that appears in the definition of $b_{k+1}$).  Therefore,
$$
\mathcal{N}(a_k,u)\leq(4k-2)\cdot|a_k|.
$$
If $u$ occurred as a subword of $a_{k+1}a_{k+1}$, such an occurrence would account for half of the letters in $a_{k+1}a_{k+1}$ and we could write
$$
a_{k+1}a_{k+1}=xuy
$$
where $|x|+|y|=|a_{k+1}a_{k+1}|/2$.  Any occurrence of $a_k$ in $a_{k+1}a_{k+1}$ would either be entirely inside of $u$, entirely inside of $x$, entirely inside of $y$, or partially overlap $x$ (or $y$) and $u$.  Any occurrence of $a_k$ that is not entirely inside $u$ must either start in $x$ or end in $y$ (but never both because $|u|>|a_k|$), so there are at most $|a_{k+1}a_{k+1}|/2$ locations where $a_k$ could occur that are not entirely inside $u$ and the density of occurrences of $a_k$ in concatenated copies of $a_ka_k$ is $1/p_k$.  Thus we can estimate the relation between the number of occurrences and the length:
\begin{eqnarray*}
\frac{\mathcal{N}(a_k,a_{k+1}a_{k+1})}{|a_{k+1}a_{k+1}|}&\leq&\frac{\text{$\#$ times $a_k$ occurs not entirely in $u$}}{|a_{k+1}a_{k+1}|}+\frac{\text{$\#$ times $a_k$ occurs in $u$}}{|a_{k+1}a_{k+1}|} \\
&\leq&\frac{1}{2p_k}+\frac{(4k-2)\cdot|a_k|}{|a_{k+1}a_{k+1}|} \\
&<&\frac{1}{2p_k}+(2k-1)\cdot\frac{|a_k|}{|a_{k+1}|}.
\end{eqnarray*}
By equation~\eqref{eq:inductive-bound}, we have that $\frac{|a_k|}{|a_{k+1}|}<\frac{1}{(4k-2)p_k}-\frac{1}{3^k\cdot|a_k|}$ and so $\frac{\mathcal{N}(a_k,a_{k+1}a_{k+1})}{|a_{k+1}a_{k+1}|}<\frac{1}{p_k}-\frac{2k-1}{3^k|a_k|}$.  On the other hand, by equation~\eqref{eq:sum}, we have $\varepsilon_k<\sum_{m=k}^{\infty}\varepsilon_m<\frac{1}{3^k}$ and so, by equation~\eqref{eq:estimate-k}, we also have $\frac{\mathcal{N}(a_k,a_{k+1}a_{k+1})}{|a_{k+1}a_{k+1}|}\geq\frac{1}{p_k}-\frac{2k-1}{3^k|a_k|}$, a contradiction.  Therefore $u$ cannot occur as a subword of $vv=a_{k+1}a_{k+1}$.  The last possiblity we need to consider is when $u=w_{(2k-2,k+1)}$.  Then $u$ is the concatenation of $(2k+1)n_{k+1}+2(k+1)$ copies of $a_k$.  But the largest power $m$ such that $a_k^m$ occurs as a subword of $a_{k+1}a_{k+1}$ satisfies $m\leq2n_{k+1}+2$,  because any larger power would force an occurrence of and element of $\mathcal{W}_k\setminus\{a_k\}$ (that occurs in $a_{k+1}a_{k+1}$) as a subword of $a_k^m$, contradicting~\eqref{claim:new}.  Thus we have the statement for $v=a_{k+1}$,
and the argument for $v=b_{k+1}$ is similar, reversing the roles played by $a_k$ and $b_k$.
\end{proof}

\subsubsection{Construction of the shift $(X, \sigma)$}
\label{sec:def-of-Z-system}
For all $k>2$,
note
that $a_{k+1}$ begins and ends with the word $a_k$.  Therefore there is a unique $x\in\{0,1\}^{\Z}$ such that for all $k>2$ we have
$$
x_{\scriptscriptstyle{1}}x_{\scriptscriptstyle{2}}\dots x_{\scriptscriptstyle{|a_k|}}=x_{\scriptscriptstyle{-|a_k|+1}}x_{\scriptscriptstyle{-|a_k|+2}}\dots x_{\scriptscriptstyle{0}}=a_k.
$$
Set $X$ to be the orbit closure of $x$ in $\{0,1\}^{\Z}$ and for the remainder of this section, we study the properties of the shift $(X, \sigma)$.

\begin{proposition}
\label{prop:structure}
If $k>2$, any $y\in X$ can be written as a bi-infinite concatenation of words constructed on level $k$.  In the special case that $y=x$, more holds: whenever $n\equiv1\pmod{|a_k|}$ the subword $x_{n}x_{n+1}\dots x_{n+|a_k|-1}$ is a level-$k$ word.

Furthermore, if $n\equiv1\pmod{|a_k|}$ and $x_n x_{n+1}\dots x_{n+2|a_k|-1}=uv$ is the concatenation of two level-$k$ words, then either $u=v$ or $uv$ is one of the following:
	\begin{enumerate}
	\item $a_kw_{(i,k)}$ or $w_{(i,k)}a_k$ for some $1\leq i\leq2k-2$; \label{c1}
	\item $b_kw_{(i,k)}$ or $w_{(i,k)}b_k$ for some $1\leq i\leq2k-2$; \label{c2}
	\item $a_kb_k$; \label{c3}
	\item $b_ka_k$. \label{c4}
	\end{enumerate}

Moreover, for any $y\in X$ there exists $0\leq r<|a_k|$ and a sequence $\{n_m\}_{m=1}^{\infty}$ such that $n_m\equiv r\pmod{|a_k|}$ for all $m\in \N$ and $\lim_{m\to\infty}\sigma^{n_m}x=y$.  With this value of $r$ fixed, whenever $n\equiv r+1\pmod{|a_k|}$ then $y_ny_{n+1}\dots y_{n+2|a_k|-1}=uv$ is a concatenation of two level-$k$ words and either $u=v$ or $uv$ is one of the forms listed in~\eqref{c1}, \eqref{c2}, \eqref{c3}, \eqref{c4}.
\end{proposition}

\begin{proof}
It suffices to prove this for the transitive point $x$ used to define $(X, \sigma)$, as then the result for general $y\in X$ follows by noting that every level-$k$ word has the same length and $y$ is in the orbit closure of $x$.

By construction, all level $k$ words have the same length and for any $s>0$, $a_{k+s}$ is a concatenation of level-$(k+s-1)$ words.  Therefore, $a_{k+s}$ is a concatenation of level-$k$ words.
For each $s$, the word $x_{-|a_{k+s}|+1}\dots x_{|a_{k+s}|}=a_{k+s}a_{k+s}$ and so $x_{-|a_{k+s}|+1}\dots x_{|a_{k+s}|}$ can be written as a concatenation of level-$k$ words and $x_nx_{n+1}\dots x_{n+|a_k|-1}$ is a level-$k$ word for all $n\equiv1\pmod{|a_k|}$ with $-|a_{k+s}|<n\leq|a_{k+s}|$.
This holds for all $s$, and so $x$ can be written as a concatenation of level-$k$ words and $x_nx_{n+1}\dots x_{n+|a_k|-1}$ is such a word whenever $n\equiv1\pmod{|a_k|}$.

We next turn to the second statement on the form of concatenated level-$k$ words. We consider a level-$(k+s)$ word $w$ and proceed by induction on $s>0$.  It suffices to show that when $w$ is written as a concatenation of level-$k$ words and $uv$ is the concatenation of two adjacent level-$k$ words that occur when $w$ is parsed in this way, then either $u=v$ or $uv$ one of forms given in~\eqref{c1}--\eqref{c4}.
By the definition of the level-$(k+1)$ words, the result clearly holds for $s=1$.  For $s=2$, this follows immediately by writing  the level-$2$ words as concatenations of level-$1$ words.  Inductively, assume
this holds for all level-$(k+s)$ words for some $s>1$, and we consider the level-$(k+s+1)$ words.  When $1\leq i\leq2(k+s+1)-4$, the word $w_{(i,k+s+1)}$ is the self-concatenation of $w_{(i,k+s)}$ a large number of times.  But, since $s>1$, $w_{(i,k+s)}$ is itself a self-concatenation of level-$(k+s-1)$ words, this means that any two adjacent level-$k$ words that occur in $w_{(i,k+s)}w_{(i,k+s)}$ already occurred within $w_{(i,k+s-1)}$ (if $i\leq2(k+s)-3$) or within $a_{k+s-1}$ or $b_{k+s-1}$.  By induction, the result holds for all $s\geq 1$.

For the final statement of this proposition, note that if $y\in X$, then $y\in\overline{\mathcal{O}(x)}$ and so exists a sequence $\{n_m\}_{m=1}^{\infty}$ such that $\lim_{m\to\infty}\sigma^{n_m}x=y$.  Without loss of generality, passing to a subsequence if necessary, we can assume that there exists $0\leq r<|a_k|$ such that $n_m\equiv r\pmod{|a_k|}$ for all $m$.  So whenever $n\equiv r+1\pmod{|a_k|}$ then $y_n y_{n+1}\dots y_{n+2|a_k|-1}$ is the same as a subword of $x$ that begins at some location congruent to $1\pmod{|a_k|}$.
\end{proof}


\begin{proposition}
\label{prop:Z-CAM}
The system $(X, \sigma)$ is a CAM subshift.
\end{proposition}
\begin{proof}
Let $x\in X$ denote the transitive point  used to define the system.
We start by showing that  the periodic points are dense in $X$.  For any $k>2$, the word $a_ka_k=x_{-|a_k|+1}\dots x_{|a_k|}$ and so the orbit of any periodic point that contains $a_ka_k$ as a subword is  within distance $1/2^{|a_k|+1}<1/2^{2^{k-1}+1}$ of the point $x$.  For each fixed $k$ and any $s>0$, the word $w_{(2k-1,k+s)}$ is the self-concatenation of many copies of $a_k$.  Therefore there is a periodic point $p_k\in X$ that is the bi-infinite self-concatenation of copies of $a_k$.  It follows that $x=\lim_{k\to\infty} p_k$.  Since $x$ is a transitive point and is approximated arbitrarily well by periodic points, it then follows that the periodic points are dense in $X$.

We next check that if $y\in X$ is not periodic, then it has a dense orbit.  If, for all $k>2$, the word $a_ka_k$ occurs as a subword of $y$, then the orbit of $y$ is dense since points in its orbit would approximate $x$ arbitrarily well.
We proceed by contradiction, assuming the existence of some $y\in X$ that is not periodic and does not have a dense orbit.  Thus we can assume that there is some $k> 2$ such that $a_ka_k$ does not occur as a subword of $y$.

The word $a_{k+1}$ does not occur as a subword of $y$, because this would force an occurrence of $a_ka_k$ in $y$.  By Proposition~\ref{prop:structure}, $y$ can be written as a bi-infinite concatenation of words constructed on level $k+2$ and, therefore, can be written as a bi-infinite concatenation of $w_{(1,k+2)},\dots,w_{(2k-2,k+2)}$ as well as $w_{(2k,k+2)}$ and $w_{(2k+2,k+2)}$.  We omit the words $w_{(2k-1,k+2)}, w_{(2k+1,k+2)}$, $a_{k+2}$, and $b_{k+2}$ from the list of possible level $k+2$ words because:
	\begin{itemize}
	\item $w_{(2k-1,k+2)}$ is the self-concatenation of many copies of $a_k$;
	\item $w_{(2k+1,k+2)}$ is the self-concatenation of many copies of $a_{k+1}$, which contains $a_ka_k$ as a subword;
	\item $a_{k+2}$ has $a_{k+1}$ as a subword and $a_{k+1}$ has $a_ka_k$ as a subword;
	\item $b_{k+2}$ has $a_{k+1}$ as a subword and $a_{k+1}$ has $a_ka_k$ as a subword.
	\end{itemize}
Again applying to Proposition~\ref{prop:structure} to the words constructed on level $k+2$, whenever $uv$ are two adjacent level $k+2$ words that occur in $x$, we must have $u=v$ since none of the other forms listed in Proposition~\ref{prop:structure}
can be written only using the words $w_{(1,k+2)},\dots,w_{(2k-2,k+2)}$ and $w_{(2k+2,k+2)}$.  But then it follows that $y$ is periodic, a contradiction.
Thus we have shown that $(X, \sigma)$ is weakly CAM and is expansive, and so by Proposition~\ref{prop:weak-to-strong},
the system $(X, \sigma)$ is CAM.
\end{proof}

\subsection{Invariant measures on the system}
\label{sec:measures}
We continue letting $(X, \sigma)$ denote the system defined in Section~\ref{sec:def-of-Z-system}.
\begin{proposition}
\label{prop:two-measure}
The system $(X, \sigma)$ supports two nonatomic ergodic measures.
\end{proposition}
\begin{proof}
Our goal is to construct two (not necessarily ergodic) measures, $\mu_a$ and $\mu_b$, supported on $X$, and show the following:
\begin{enumerate}        \item $\mu_a\neq\mu_b$ (which we do by showing $\mu_a([0])\neq\mu_b([0])$);
    \item almost every measure in the ergodic decomposition of $\mu_a$ is nonatomic;
    \item almost every measure in the ergodic decomposition of $\mu_b$ is nonatomic.
\end{enumerate}
It follows from these facts that there exist two distinct, ergodic, nonatomic measures supported on $X$ and, therefore, that CAM systems are capable of having this property.

Let $x\in X$ be the infinite word defined by the condition that for all $k\geq 1$, we have
$$
a_k=x_1 \dots x_{|a_k|}=
x_{-|a_k|+1} \dots x_0
$$
and let $x^{\prime}\in X$ be the infinite word defined by the condition that for all $k\geq 1$,  we have
$$
b_k=x^{\prime}_1 \dots x^{\prime}_{|a_k|} =x^{\prime}_{-|a_k|+1} \dots x^{\prime}_0.
$$
For each $k\geq3$, set
$$
\nu_{k,a}:=\frac{1}{2|a_k|}\sum_{m=-|a_k|+1}^{|a_k|}\delta_{\sigma^m(x)}
$$
and set
$$
\nu_{k,b}:=\frac{1}{2|b_k|}\sum_{m=-|b_k|+1}^{|b_k|}\delta_{\sigma^m(x^{\prime})}.
$$
Let $\mu_a$ be a weak*-limit point of $\{\nu_{k,a}\}_{k=1}^{\infty}$ and let $\mu_b$ be a weak*-limit point of $\{\nu_{k,b}\}_{k=1}^{\infty}$.  Then for all $k\geq 1$, we have
$$
\nu_{k,a}([0])=\frac{\mathcal{N}(0,a_k)}{|a_k|}\leq\sum_{j=1}^k\varepsilon_j<\frac{1}{3} \hspace{0.2 in} \text{and} \hspace{0.2 in} \nu_{k,b}([1])=\frac{\mathcal{N}(1,b_k)}{|b_k|}\leq\sum_{j=1}^k\varepsilon_j<\frac{1}{3}
$$
and so in particular it follows that
$\mu_a([0])\leq\frac{1}{3}$ and $\mu_b([1])\leq\frac{1}{3}$.  Let
$$
\mu_a=\int_X\mu_a^xd\mu_a(x) \hspace{0.2 in} \text{and} \hspace{0.2 in} \mu_b=\int_X\mu_b^xd\mu_b(x)
$$
be the ergodic decompositions of $\mu_a$ and $\mu_b$, respectively.  Then we have
	\begin{eqnarray}
 \label{eq:biga}
	\mu_a\left(\left\{x\colon\mu_a^x([1])\geq\frac{2}{3}\right\}\right)&>&0; \\
 \label{eq:bigb}
	\mu_b\left(\left\{x\colon\mu_b^x([0])\geq\frac{2}{3}\right\}\right)&>&0.
	\end{eqnarray}
Let $p\in X$ be a periodic point.  By construction of $X$, there is some $k\in\mathbb{N}$ and a level-$k$ word, $u$,  such that $p$ is a shift of the word $\cdots uuuu\cdots$ (because arbitrarily long segments of $p$ have to occur in $a_{k+t}$ as $t\to\infty$).
Fix $\varepsilon>0$ and find $m>k$ such that $\sum_{j=m}^{\infty}\varepsilon_j<\varepsilon$.  Let $\tilde{u}$ be the level-$m$ word that is $|u|$-periodic and is the self-concatenation of a large number of copies of $u$.  Since $m>k$, we have that $\tilde{u}\neq a_m$ and so for all $t\geq0$,
it follows from Equation~\eqref{eq:fraction4} that
$$
\frac{\mathcal{N}(\tilde{u},a_{m+t}a_{m+t})}{|a_{m+t}a_{m+t}|}<\frac{1}{|\tilde{u}|(2|u|-1)}\sum_{j=m}^{\infty}\varepsilon_j<\frac{\varepsilon}{|\tilde{u}|(2|u|-1)}.
$$
Therefore
\[\nu_{a,m+t}([\tilde{u}])<\frac{\varepsilon}{|\tilde{u}|(2|u|-1)}+\frac{|\tilde{u}|-1}
{|a_{m+t}a_{m+t}|}
\]
for all $t\geq0$, as
the second term on the right hand side is from possible occurrences of $\tilde{u}$ in $x$ that are within distance $|\tilde{u}|$ of the right edge of the rightmost copy of $a_{m+t}$.  Thus it follows that
$\mu_a([\tilde{u}])\leq\varepsilon/|\tilde{u}|(2|u|-1)$.  Similarly, using the same argument but replacing $a_m$ by $b_m$, it follows that
$\mu_b([\tilde{u}])\leq\varepsilon/|\tilde{u}|(2|u|-1)$.  Repeating this argument for any $\varepsilon>0$, we conclude that
$$
\lim_{c\to\infty}\mu_a([u^c])=0 \hspace{0.2 in} \text{and} \hspace{0.2 in} \lim_{c\to\infty}\mu_b([u^c])=0.
$$
This implies that $\mu_a(\{x\colon\mu_a^x\text{ is atomic and concentrated on $p$}\})=0$ and, similarly, that $\mu_b(\{x\colon\mu_b^x\text{ is atomic and concentrated on $p$}\})=0$.
Since the choice of $p$ was arbitrary,
it follows that $\mu_a$-almost every ergodic component of $\mu_a$ is nonatomic and $\mu_b$-almost every ergodic component of $\mu_b$ is nonatomic.
However by~\eqref{eq:biga}, there is a set of ergodic components of $\mu_a$ of positive $\mu_a$-measure that give measure at least $2/3$ to the cylinder set $[1]$, and hence there is at least one nonatomic ergodic measure that does.  Similarly, by~\eqref{eq:bigb} there is a nonatomic ergodic component of $\mu_b$ that gives measure at most $1/3$ to the cylinder set $[1]$, and so there is at least one nonatomic ergodic measure that does.
It follows that these two nonatomic, ergodic measures are distinct.
\end{proof}

Combining Propositions~\ref{prop:Z-CAM} and~\ref{prop:two-measure}, we have thus completed the proof of Theorem~\ref{th:Z-CAM}
Some open quesitions and directions for exploration about the system defined in~\ref{sec:def-of-Z-system} are contained in Section~\ref{sec:further}.

\section{{$\Z^d$}-CAM system}
\label{sec:higher}
\subsection{The setup}
In Section~\ref{sec:Zcamsystem}, we showed the existence of a symbolic $\mathbb{Z}$-CAM system and further showed that such a system can support two distinct nonatomic, ergodic probability measures. In this section we generalize this result to $\Z^d$.
The inductive construction is similar to that in one dimension, but the added dimensions complicate the counting of ways in which configurations can overlap, and so in the initial levels of the construction we include all details.  When the proofs become close enough that changes are mainly changing notation, we omit the details.

Similar to the one dimensional case, we construct a $\Z^d$-subshift on the alphabet $\{0,1\}$, and the higher dimensional setting necessitates changes in some of the definitions.
If $\mathcal{S}\subseteq\Z^d$, we refer to $\CS$ as a {\em shape} and define an {\em $\mathcal{S}$-word} to be a function $w\colon\mathcal{S}\to\{0,1\}$.
When $w$ is an $\CS$-word, we let $[w]$ denote any function whose domain is a translate of $\mathcal{S}$ and we say  $[w]$ {\em  has the same shape} as $w$; we let
$|w|$ denote the number of integer points in $\mathcal{S}$.

For the special case that $\mathcal{S}$ is rectangular, we make use of a notion of subwords.
Assume that $\mathcal{S}=\prod_{i=1}^d\{1,2,\dots,m_i\}$ for some $m_1,\dots,m_d\in\N$, $\mathcal{T}=\prod_{i=1}^d\{1,2,\dots,m_i^{\prime}\}$ with $m_i^{\prime}\leq m_i$ for $i=1, \dots, d$, let $w\colon\mathcal{S}\to\mathcal{A}$ be an $\mathcal{S}$-word and let $v\colon\mathcal{T}\to\mathcal{A}$ be a $\mathcal{T}$-word.  We say that $v$ occurs as a {\em subword} of $w$ if there exist $k_1\leq m_1-m_1^{\prime},\dots,k_d\leq m_d-m_d^{\prime}$ such that $v(x_1,\dots,x_d)=w(x_1+k_1,\dots,x_d+k_d)$ for all $(x_1,\dots,x_d)\in\mathcal{T}$.

For fixed $d,e_1,\dots,e_d,n\geq1$ and function $w\colon\{1,2,\dots,n\}^d\to\{0,1\}$,
define the
{\em $(e_1\times e_2\times\cdots\times e_d)$-fold self-concatenation of $w$} to be the function
$$
w^{(e_1,\dots,e_d)} \colon\prod_{i=1}^d\{1,2,\dots,e_in\}\to\{0,1\}
$$
defined by
$$
w^{(e_1,\dots,e_d)}(x_1,\dots,x_d):=w\left(x_1\text{ }(\mathrm{mod}\text{ }n), x_2\text{ }(\mathrm{mod}\text{ }n),\dots,x_d\text{ }(\mathrm{mod}\text{ }n)\right).
$$
For convenience, when $d$ is
clear from the context, we write $w^{(e)}$ as shorthand for  $w^{(e,e,\dots,e)}$.

If $u_1,\dots,u_k\colon\{1,2,\dots,n\}^d\to\{0,1\}$ and $e\geq2k+4$, define
$$
\mathcal{P}[u_1,\dots,u_k\mid w,e]\colon\{1,2,\dots,(2e+1)n\}^d\to\{0,1\}
$$
to be the {\em postcard function}, where we consider $u_1,\dots,u_k$ to be the {\em stamp}  (thought of as in the lower left corner)
defined by:

\begin{enumerate}
\item[Case 1] (the stamp): if $(x_1,\dots,x_d)\in\{2mn+1,\dots,(2m+1)n\}\times\{2n+1,\dots,3n\}^{d-1}$ for some $1\leq m\leq k$,  define
$$
\mathcal{P}[u_1,\dots,u_k\mid w,e](x_1,\dots,x_d):=u_m(x_1-2mn,x_2-2n,\dots,x_d-2n).
$$

\item[Case 2] (the rest of the postcard): if $(x_1,\dots,x_d)\notin\{2mn+1,\dots,(2m+1)n\}\times\{2n+1,\dots,3n\}^{d-1}$ for any $1\leq m\leq k$, we define
$$
\mathcal{P}[u_1,\dots,u_k\mid w,e](x_1,\dots,x_d):=w^{(e)}(x_1,\dots,x_d)
$$
\end{enumerate}





\begin{figure}
\begin{tikzpicture}[font=\sffamily, scale=0.8, every node/.style={scale=0.8}]

  \tikzstyle{normal} = [rectangle, draw, fill=white, minimum width=1cm, minimum height=.7cm]
  \tikzstyle{highlight} = [rectangle, draw, fill=yellow!50, minimum width=1cm, minimum height=.7cm]

  \newcommand{\sequence}[2]{
    \ifnum#1=2 \node[highlight] at (#1*1,0) {#2}; \fi
    \ifnum#1=4 \node[highlight] at (#1*1,0) {#2}; \fi
    \ifnum#1=2\else\ifnum#1=4\else\node[normal] at (#1*1,0) {#2};\fi\fi
  }

  \foreach \x/\y in {0/$w$,1/$w$,2/$u_{1}$,3/$w$,4/$u_{2}$,5/$w$,6/$w$,7/$w$,8/$w$,9/$w$,10/$w$,11/$w$,12/$w$}{
    \sequence{\x}{\y}
  }
\end{tikzpicture}
\caption{The postcard $\mathcal{P}[u_1, u_2 \mid w, 6]\colon \{1, \dots, 13|w|\}\to\{0,1\}$}
\label{figure:1d}
\end{figure}
Examples of postcards in  dimensions $1$ and $2$ are given in Figures~\ref{figure:1d} and~\ref{figure:2d}.

We define $\mathcal{N}(u,v)$ to be the number of times $u$ occurs as a subword of $v$.  A useful feature of the postcard function is that if $u$ never occurs as a subword of $w^{(2)}$, then for all $1\leq i\leq k$ we have
$$
\lim_{e\to\infty}\frac{\mathcal{N}(u,\mathcal{P}[u_1,\dots,u_k\mid w,e])}{|\mathcal{P}[u_1,\dots,u_k\mid w,e]|}=0.
$$

\begin{figure}
\centering
\begin{tikzpicture}[font=\sffamily, scale=0.8, every node/.style={scale=0.8}]
  \tikzstyle{normal} = [rectangle, draw, fill=white, minimum width=1cm, minimum height=.7cm]
  \tikzstyle{highlight} = [rectangle, draw, fill=yellow!50, minimum width=1cm, minimum height=.7cm]

  \newcommand{\drawSequenceRow}[2]{
    \foreach \x/\y in {0/$w$,1/$w$,2/$w$,3/$w$,4/$w$,5/$w$,6/$w$,7/$w$,8/$w$}{
      \pgfmathsetmacro{\ypos}{#1*-.69} 
      \ifnum#2=1
        \ifnum\x=2 \node[highlight] at (\x*1,\ypos) {\y}; \fi
        \ifnum\x=4 \node[highlight] at (\x*1,\ypos) {\y}; \fi
        \ifnum\x=2\else\ifnum\x=4\else\node[normal] at (\x*1,\ypos) {\y};\fi\fi
      \else
        \node[normal] at (\x*1,\ypos) {\y};
      \fi
    }
  }

  \newcommand{\drawSequenceRowMod}[2]{
    \foreach \x/\y in {0/$w$,1/$w$,2/$u_{1}$,3/$w$,4/$u_{2}$,5/$w$,6/$w$,7/$w$,8/$w$}{
      \pgfmathsetmacro{\ypos}{#1*-.69} 
      \ifnum#2=1
        \ifnum\x=2 \node[highlight] at (\x*1,\ypos) {\y}; \fi
        \ifnum\x=4 \node[highlight] at (\x*1,\ypos) {\y}; \fi
        \ifnum\x=2\else\ifnum\x=4\else\node[normal] at (\x*1,\ypos) {\y};\fi\fi
      \else
        \node[normal] at (\x*1,\ypos) {\y};
      \fi
    }
  }

  \drawSequenceRow{0}{0} 
  \drawSequenceRow{1}{0} 
  \drawSequenceRow{2}{0} 
  \drawSequenceRow{3}{0} 
  \drawSequenceRow{4}{0} 
  \drawSequenceRow{5}{0} 
  \drawSequenceRowMod{6}{1} 
  \drawSequenceRow{7}{0} 
  \drawSequenceRow{8}{0} 

\end{tikzpicture}
\caption{The postcard $\mathcal{P}[u_1, u_2 \mid w, 4]\colon \{1, \dots, 9|w|\}^{2}\to\{0,1\}$}
\label{figure:2d}
\end{figure}

\subsection{The construction} This construction proceeds recursively, in levels, and is highly reminiscent of the construction of a $\Z$-CAM subshift.  The alphabet is $\mathcal{A}=\{0,1\}$, and again, all  parameters are tuned after the fact to give the resulting system its properties.  The first three levels are special and so we construct them explicitly, and the recursive formula starts with level four.  As previously, the second coordinate of a subscript always denotes the level.

\subsubsection{The frequency sequence}  Pick a sequence
$\{\varepsilon_k\}_{k=1}^{\infty}$ of non-negative real numbers that satisfy
$$
\sum_{k=N}^{\infty}\varepsilon_k<\frac{1}{3^{d+N-1}}.
$$
for all $N\in\N$.  We  use this sequence throughout our construction.

\subsubsection{Level 1 (introduction of the words $w$)} There are $2$ words on this level which are both functions from $\{1\}^d$ to $\mathcal{A}$, and we set
$$
w_{(1,1)}(1,\dots,1)=0 \quad \text{ and } \quad w_{(2,1)}(1,\dots,1)=1
$$

\subsubsection{Level 2 words (introduction of the words $a$ and $b$)} There are $4$ words on this level,
playing two distinct roles as distinguished by two types: we have periodic words $w$ and density words $a,b$.  This level also introduces a parameter $n_2>1$ that controls the densities (we  explain the meaning of this momentarily). All words at this level are functions $\{1,2,\dots,n_2\}^d\to\mathcal{A}$.  We define:
\begin{align*}
w_{(1,2)}&=  w_{(1,1)}^{(n_2)} &   \text{self-concatenation of copies of $w_{(1,1)}$} \\  \\
w_{(2,2)}&=w_{(2,1)}^{(n_2)} &  \text{self-concatenation of  copies of $w_{(2,1)}$} \\  \\
a_2&=\mathcal{P}[w_{(2,1)}|w_{(1,1)},{n_2}] & \text{
all $w_{(1,1)}$ except a single copy of $w_{(2,1)}$} \\
b_2&=\mathcal{P}[w_{(1,1)}|w_{(2,1)},{n_2}] & \text{
all $w_{(2,1)}$ except a single copy of $w_{(1,1)}$}
\end{align*}
The parameter $n_2$ is chosen sufficiently large such that $\mathcal{N}(1,a_2)/|a_2|<\varepsilon_1$ and $\mathcal{N}(0,b_2)/|b_2|<\varepsilon_1$.  In other words, $a_2$ consists almost entirely  of $0$'s and $b_2$  almost entirely  of $1$'s.  More explicitly, $a_2,b_2\colon\{1,2,\dots,n_2\}^d\to\mathcal{A}$ and the only point $(x_1,\dots,x_d)$ in the domain of $a_2$ where $a_2(x_1,\dots,x_d)=1$ is $(x_1,\dots,x_d)=(3,3,\dots,3)$.  Similarly for $b_2(x_1,\dots,x_d)=0$.

We call the four words introduced at this point the {\em level-2 words} and we write $\mathcal{W}_2=\{w_{(1,2)},w_{(2,2)},a_2,b_2\}$.
For use in verifying properties of the words constructed at higher levels, we describe the basic properties of the words constructed at level 2.
\begin{proposition}\label{propd:level1}
For all distinct words $u,v\in\mathcal{W}_2$, the word $u$ does not occur as a subword of $v^{(2)}$.
\end{proposition}
\begin{proof}
First we establish the statement when $v\in\{w_{(1,2)},w_{(2,2)}\}$ and $u\in\mathcal{W}_2$ is distinct from $v$.
Neither $a_2$ nor $b_2$ occurs as a subword of $w_{(i,2)}^{(2)}$ for $i=1,2$, because $w_{(i,2)}^{(2)}$ is a constant function, whereas $a_2$ and $b_2$ are not.  The word $w_{(1,2)}$ does not occur as a subword of $w_{(2,2)}^{(2)}$, since the symbol $0$ does not occur in $w_{(2,2)}$ but does occur in $w_{(1,2)}$ and similarly $w_{(2,2)}$ does not occur as a subword of $w_{(1,2)}^{(2)}$.  This establishes the statement when $v\in\{w_{(1,2)},w_{(2,2)}\}$.

Next we establish the statement when $v\in\{a_2,b_2\}$ and $u\in\mathcal{W}_2$ is distinct from $v$.
Neither $w_{(1,2)}$ nor $w_{(2,2)}$  occurs as a subword of $a_2^{(2)}$, as the largest $m$ for which there is a constant subword of shape $\{1,2,\dots,m\}^d$ in $a_2^{(2)}$ satisfies $m<n_2$,
whereas any occurrence of $w_{(i,2)}$ forces a constant subword of
shape $\{1,2,\dots,n_2\}^d$.  Similarly, neither can occur as a subword of $b_2^{(2)}$.  We claim that the word $a_2$ does not occur as a subword of $b_2^{(2)}$ (the argument that $b_2$ does not occur as a subword of $a_2^{(2)}$ is analogous).  By construction, $\mathcal{N}(1,a_2)/|a_2|<\varepsilon_1$ and so $\mathcal{N}(1,a_2)<\varepsilon_1|a_2|$.
Therefore $\mathcal{N}(0,a_2)\geq(1-\varepsilon_1)|a_2|$.
Similarly, using the fact that $b_2^{(2)}$ consists of concatenated copies of $b_2$, we get $\mathcal{N}(1,b_2^{(2)})\geq(1-\varepsilon_1)|b_2^{(2)}|$.  If $a_2$ occurs as a subword of $b_2^{(2)}$ then, using the fact that $|a_2|=\frac{1}{2^d}|b_2^{(2)}|$, we obtain
$$
\mathcal{N}(0,b_2^{(2)})\geq\mathcal{N}(0,a_2)\geq(1-\varepsilon_1)|a_2|=\frac{1-\varepsilon_1}{2^d}|b_2^{(2)}|.
$$
This implies that $\mathcal{N}(1,b_2^{(2)})\leq\left(1-\frac{1-\varepsilon_1}{2^d}\right)|b_2^{(2)}|=\frac{2^d-1+\varepsilon_1}{2^d}|b_2^{(2)}|$.  Combined with our earlier observation that $\mathcal{N}(1,b_2^{(2)})\geq(1-\varepsilon_1)|b_2^{(2)}|$, we have that $1-\varepsilon_1\leq\frac{2^d-1+\varepsilon_1}{2^d}$.  However, this contradicts our assumption that  $\varepsilon_1<\sum_{k=1}^{\infty}\varepsilon_k<\frac{1}{3^d}\leq\frac{1}{2^d+1}$.  Therefore the statement holds when $v\in\{a_2,b_2\}$.
\end{proof}

\subsubsection{Level $3$ words (recursive definitions of periodic and density words)}
There are $6$ words on this level,  again playing distinct roles distinguished by two types: periodic words $w$ and density words $a,b$.  Again we introduce a parameter $n_3 >1 $ to control the densities.   Define the words by setting
\begin{align*}
w_{(1,3)}&=w_{(1,2)}^{(5n_3+4)}
\\ \\
w_{(2,3)}&=w_{(2,2)}^{(5n_3+4)}
\\ \\
w_{(3,3)}&=a_2^{(5n_3+4) }
\\ \\
w_{(4,3)}&=b_2^{(5n_3+4)}
\\ \\
a_3&=\mathcal{P}[w_{(1,2)},w_{(2,2)},a_2,b_2|a_2,n_3]
\\ \\
b_3&=\mathcal{P}[w_{(1,2)},w_{(2,2)},a_2,b_2|b_2,n_3],
\end{align*}
meaning that  the first four words periodize the words $w_{(1,2)}$, $w_{(2,2)}$, $a_2$, and $b_2$ (in order), the second to last word is mostly concatenated $a_2$'s with the words $\mathcal{W}_2$ as a stamp, and the last  word is mostly concatenated $b_2$'s with the words $\mathcal{W}_2$ as a stamp.
We choose $n_3$ such  that $\mathcal{N}(1,a_3)/|a_3|<\varepsilon_1+\varepsilon_2$ and $\mathcal{N}(0,b_3)/|b_3|<\varepsilon_1+\varepsilon_2$ (we note that this is possible because $\mathcal{N}(1,a_2)/|a_2|<\varepsilon_1$ and $a_3$ is mainly made of concatenated copies of $a_2$, similarly $\mathcal{N}(0,b_2)/|b_2|<\varepsilon_1$).
We further assume that $n_3$ has been chosen sufficiently large such that for $u\in\{w_{(1,2)},w_{(2,2)},b_2\}$,  we have
\begin{equation}\label{eqd:fraction1}
\frac{\mathcal{N}(u,a_3^{(2)})}{|a_3^{(2)}|}<\frac{\varepsilon_2}{|u|(2|u|-1)^d},
\end{equation}
and note that this can be done since $a_3^{(2)}$ is mainly made of concatenated copies of $a_2^{(2)}$ and  $u$ does not arise as a subword of $a_2^{(2)}$.  Using the same reasoning, this choice of $n_3$ ensures that if $u\in\{w_{(1,2)},w_{(2,2)},a_2\}$, then we have
\begin{equation}\label{eqd:fraction2}
\frac{\mathcal{N}(u,b_3^{(2)})}{|b_3^{(2)}|}<\frac{\varepsilon_2}{|u|(2|u|-1)^d}.
\end{equation}
Finally note that by choosing $n_3$ sufficiently large, we can guarantee that $|a_2|/|a_3|$ is as small as desired, and for use in the proof of the next proposition describing occurrences of words and subwords, we need a specific bound.  Let $a_2^{(\infty)}\colon\Z^d\to\mathcal{A}$ be the infinite self-concatenation of $a_2$ given by:
$$
a_2^{(\infty)}(x_1,\dots,x_d):=a_2(x_1\text{ (mod $n_2$)},\dots,x_d\text{ (mod $n_2$)}).
$$
Let $\mathcal{V}\subseteq\Z^d$ be the set of all period vectors of $a_2^{(\infty)}$: $(v_1,\dots,v_d)\in\mathcal{V}$ if and only if $a_2^{(\infty)}(x_1,\dots,x_d)=a_2^{\infty}(x_1+v_1,\dots,x_d+v_d)$ for all $(x_1,\dots,x_d)\in\Z^d$.  Then $\mathcal{V}$ is a finite-index subgroup of $\Z^d$; let $p$ be the index of $\mathcal{V}$ in $\Z^d$.  Note that at any location where $a_2$ occurs as a subword of $a_2^{(2)}$, say on the set $\prod_{i=1}^d\{y_i+1,y_i+2,\dots,y_i+n_2\}$, we get $(y_1,\dots,y_d)\in\mathcal{V}$.  Therefore, the number of times $a_2$ occurs as a subword of $a_2^{(2)}$ satisfies:
\begin{multline}\label{eqd:multiplicity}
\frac{n_2^d}{p}\leq\left|\{1,\dots,2n_2\}^d\cap\mathcal{V}\right|\leq\left|\{1,\dots,2n_2-1\}^d\cap\mathcal{V}\right|+dn_2^{d-1} \\ =\frac{n_2^d}{p}+dn_2^{d-1}=n_2^d\cdot\left(\frac{1}{p}+\frac{d}{n_2}\right).
\end{multline}
We call the words constructed on this level the {\em level-3 words} and set $\mathcal{W}_3$ to be the set of all level-3 words.

\begin{proposition}\label{propd:level2}
For any choice of  distinct words $u,v\in\mathcal{W}_3$, the word $u$ does not occur as a subword of $v^{(2)}$.
\end{proposition}
\begin{proof}
We check that for each $v\in\mathcal{W}_3$ and  $u\in\mathcal{W}_3\setminus\{v\}$, the word $u$ does not occur as a subword of $v^{(2)}$, checking cases depending on the type of word $v$.

First consider when  $v=w_{(i,3)}$ for some $i\in\{1,2\}$.   We consider two cases, depending on the choice of $u$. First suppose that $u\in\{w_{(3,3)},w_{(4,3)},a_3,b_3\}$.  By Proposition~\ref{propd:level1}, neither $a_2$ nor $b_2$ occurs as a subword of $w_{(i,2)}^{(2)}$.  The word $v^{(2)}=w_{(i,3)}^{(2)}$ is the self-concatenation of a large number of copies of $w_{(i,2)}$ and so neither $a_2$ nor $b_2$ occurs as a subword of $v^{(2)}$.  But at least one of $a_2$ and $b_2$ occurs as a subword of $u$, and so $u$ cannot be a subword of $v^{(2)}$.
Next suppose $u=w_{(j,3)}$ where $j\in \{1,2\}$ and $j\neq i$.  The word $u=w_{(j,3)}$ is not a subword of $v^{(2)}=w_{(i,3)}^{(2)}$ because $w_{(j,2)}$ is not a subword of $w_{(i,2)}^{(2)}$ and, as noted, $v^{(2)}$ is the self-concatenation of copies of $w_{(i,2)}$ while $w_{(j,3)}$ is the self-concatenation of copies of $w_{(j,2)}$.  We deduce that
$u$ does not occur as a subword of $v^{(2)}$ for any $u\in\mathcal{W}_3\setminus\{v\}$.  Thus the statement holds when $v\in\{w_{(1,3)},w_{(2,3)}\}$.

Next consider when $v=w_{(i,3)}$ for some $i\in\{3,4\}$.  Again we have two cases, depending on the choice of $u$.
First suppose $u\in\{w_{(1,3)},w_{(2,3)},a_3,b_3\}$.
By Proposition~\ref{propd:level1}, neither $w_{(1,2)}$ nor $w_{(2,2)}$ occurs as a subword of $a_2^{(2)}$ or of $b_2^{(2)}$.  The word $v^{(2)}=w_{(i,3)}^{(2)}$ is the self-concatenation of a large number of copies of $a_2$ or of $b_2$, and so neither $w_{(1,2)}$ nor $w_{(2,2)}$ occurs as a subword of $v^{(2)}$.
However at least one of $w_{(1,2)}$ and $w_{(2,2)}$ occurs in $u$, and so $u$ is not a subword of $v^{(2)}$.  We next consider when $j\in\{3,4\}\setminus\{i\}$ and let $u=w_{(j,3)}$.  Then one of the words $u$ and $v$ is the self-concatenation of many copies of $a_2$ and the other is the self-concatenation of many copies of $b_2$.  By Proposition~\ref{propd:level1}, $a_2$ is not a subword of $b_2^{(2)}$ and $b_2$ is not a subword of $a_2^{(2)}$, and so $u$ is not a subword of $v^{(2)}$.  This establishes the statement when $v\in\{w_{(3,3)},w_{(4,3)}\}$.

Finally consider when $v\in\{a_3,b_3\}$.  We give the argument when $v=a_3$, and the argument for $b_3$ is similar. Again, we have two cases, depending on the choice of $u$.
First suppose $u\in\{w_{(1,3)},w_{(2,3)},w_{(4,3)},b_3\}$.  By~\eqref{eqd:fraction1}, we know that $\mathcal{N}(x,a_3^{(2)})/|a_3^{(2)}|<\varepsilon_2/|x|(2|x|-1)^d$ for $x\in\mathcal{W}_2\setminus\{a_2\}$.  For any such $x$ and any particular occurrence of $x$ in $a_3^{(2)}$, there are $(2|x|-1)^d$ subwords of $a_3^{(2)}$ that have
the shape $[x]$
and partially (or completely) overlap this occurrence of $x$.  This means that, for any $x\in\mathcal{W}_2\setminus\{a_2\}$, if we look at the collection of all locations within $a_3^{(2)}$ where $x$ occurs, we have
$$
\frac{\text{$\#$ of subwords with shape $[x]$ in $a_3^{(2)}$ that overlap an occurrence of $x$}}{|a_3^{(2)}|}<\frac{\varepsilon_2}{|x|}.
$$
Recall that $p$ is the index, in $\Z^d$, of the stabilizer subgroup of $a_2^{(\infty)}\colon\Z^d\to\{0,1\}$ (when acted on by $\Z^d$ translations) and the number of occurrences of $a_2$ as a subword of $a_2^{(2)}$ satisfies~\eqref{eqd:multiplicity}.  Since all four words in $\mathcal{W}_2$ have the same size, we deduce from~\eqref{eqd:fraction1} that
\begin{equation}
\frac{\mathcal{N}(a_2,a_3^{(2)})}{|a_3^{(2)}|}\geq\frac{1}{p}-3\cdot\frac{\varepsilon_2}{|a_2|}
\end{equation}
because $a_3^{(2)}$ is made by concatenating words in $\mathcal{W}_2$.  On the other hand, we claim that $\mathcal{N}(a_2,u)\leq3\cdot(2|a_2|-1)^d$. To check this, note that if $u=b_3$, any occurrence of $a_2$ in $u$ must partially overlap $w_{(1,2)}$, $w_{(2,2)}$ or $a_2$ in the definition of $b_3$, because it does not occur as a subword of $b_2^{(2)}$, and there are only $3\cdot(2|a_2|-1)^d$ locations that have such overlaps.  If $u\neq b_3$ then $a_2$ does not occur in $u$, by Proposition~\ref{propd:level1} and the claim follows.

Now we return to showing that $u$ does not occur as a subword of $a_3^{(2)}$.  For contradiction, suppose $u$ does occur as a subword of $a_3^{(2)}$.  Recall, from the definition, that $a_3^{(2)}$ is primarily made of concatenated copies of $a_2$ and there are three specific locations where the words $w_{(1,2)}, w_{(2,2)}, b_2$ occur instead.  Since $u$ occurs as a subword of $a_3^{(2)}$ this means $u$ is primarily made of concatenated copies of $a_2$.  In fact, allowing for the possibility that $u$ occurs within $a_3^{(2)}$ in a location that overlaps the copies of $w_{(1,2)}, w_{(2,2)}, b_2$ and that there may be no copies of $a_2$ that partially overlap these words, we can still use~\eqref{eqd:multiplicity} to estimate that
$$
\mathcal{N}(a_2,u)\geq((n_2n_3-2n_2)^d\cdot\frac{1}{p}-3\cdot(2n_2-1)^d.
$$
In more detail, recall that $u\colon\{1,2,\dots,n_2n_3\}^d\to\{0,1\}$ has a domain which is a cube in dimension $d$.  The number $(n_2n_3-2n_2)^d$ is the volume of the subcube obtained by removing the border of length $n_2$ from each side.  Then, using~\eqref{eqd:multiplicity}, $((n_2n_3-2n_2)^d\cdot\frac{1}{p}$ gives the minimum number of occurrences of $a_2$ within this subcube if we assume that the subcube is entirely populated by occurrences of $a_2^{(2)}$.  Finally we subtract off $3\cdot(2n_2-1)^d$ to account for the fact that this occurrence of $u$ might overlaps the words $w_{(1,2)}, w_{(2,2)}, b_2$ in the definition of $a_3$ and that no word of size $|a_2|$ that partially overlaps these occurrences could be $a_2$.  Nonetheless, note that this lower bound on $\mathcal{N}(a_2,u)$ is larger than our previous upper bound of $3\cdot(2|a_2|-1)^d$ provided $|a_2|/|a_3|$ is sufficiently small (and we make this assumption).

Lastly, suppose that $u=w_{(3,3)}$.
By construction, $u$ is the self-concatenation of many copies of $a_2$.  If $u$ occurred within $v^{(2)}=a_3^{(2)}$ then it would either contain or partially overlap an occurrence of $w_{(1,2)}$ which is used in the definition of $a_3$.  Since $w_{(1,2)}$ does not occur as a subword of $a_2^{(2)}$, this occurrence of $u$ cannot contain $w_{(1,2)}$ as a subword so it can only occur in a location that partially overlaps $w_{(1,2)}$.  But $|u|=|v|$ and $u$ is the periodic concatenation of copies of $a_2$, so this occurrence of $u$ overlaps the copy of $w_{(1,2)}$ in all $2^d$ copies of $a_3$ that appear in the definition of $a_3^{(2)}$.  By periodicity of $u$, $w_{(1,2)}$ occurs as a subword of $a_2^{(2)}$, which is impossible.
This establishes the statement when $v=a_3$.  The argument for $v=b_3$ is similar, with the roles played by $a_2$ and $b_2$ switched.
\end{proof}

\subsubsection{Preparing to build the level $k+1$ words (inductive assumptions)}
For $k\geq3$,  we construct the words on level $k+1$ using the words on level $k$.  Again they come in two varieties (periodic words and density words), and again we have a parameter $n_{k+1}> 1$ chosen to guarantee that the words have the desired properties.  Inductively, we assume that we have constructed the level-$k$ words: these are words $w_{(i,k)}$ for all $1\leq i\leq2k-2$, words $a_k$ and $b_k$, with all of these words having the same domain: $\{1,2,\dots,\prod_{i=1}^kn_i\}^d$, where $n_1,\dots,n_k\in\N$ are parameters we assume have been defined already, and we denote the collection of all level-$k$ words by $\mathcal W_k$.  We further assume that these constructions satisfy the following properties:
	\begin{enumerate}
	\item \label{itemd:firstk}
 For $1\leq i\leq2k-4$, we have
	$$
	w_{(i,k)}=w_{(i,k-1)}^{(n_k)}.
	$$
	\item We have
	$$
	w_{(2k-3,k)}=a_{k-1}^{(n_k)}.
	$$
	\item We have
	$$
	w_{(2k-2,k)}=b_{k-1}^{(n_k)}.
	$$
	\item We have
	$$
	a_k=\mathcal{P}[w_{(1,k-1)},w_{(2,k-1)},\dots,w_{(2k-4,k-1)},a_{k-1},b_{k-1}|a_{k-1},n_k].
	$$
	\item
 \label{itemd:lastk}
 We have
	$$
	b_k=\mathcal{P}[w_{(1,k-1)},w_{(2,k-1)},\dots,w_{(2k-4,k-1)},a_{k-1},b_{k-1}|b_{k-1},n_k].
	$$
	\end{enumerate}
We further assume that for any word $u$ constructed on some level $m$ with $m<k$, other than when $u=a_m$, we have
\begin{equation}\label{eqd:fraction4}
\frac{\mathcal{N}(u,a_k^{(2)})}{|a_k^{(2)}|}<\frac{1}{|u|(2|u|-1)^d}\sum_{j=m}^{k-1}\varepsilon_j
\end{equation}
and that for any $m<k$ and any level-$m$ word $u$ other than $u=b_m$, we have
\begin{equation*}
\frac{\mathcal{N}(u,b_k^{(2)})}{|b_k^{(2)}|}<\frac{1}{|u|(2|u|-1)^d}\sum_{j=m}^{k-1}\varepsilon_j.
\end{equation*}
Finally we assume that
\begin{equation}\label{claimd:new}
\text{for any $u,v\in\mathcal{W}_k$ with $u\neq v$, the word $u$ does not occur as a subword of $v^{(2)}$.}
\end{equation}

\subsubsection{Level-$(k+1)$ words (recursive definition)}
We apply the analogous procedure used for  level $3$  and  construct the periodic and density words, this time assuming the properties given in~\eqref{itemd:firstk} through~\eqref{itemd:lastk}.  We define a parameter, $n_{k+1}$, whose properties are discussed below and we set the following notation.
	\begin{enumerate}
	\item \label{itemd:firstk+1}
 For  $1\leq i\leq2k-2$, define the first group of periodic words  by setting
	$$
	w_{(i,k+1)}=w_{(i,k)}^{(n_{k+1})}.
	$$
	\item Define the next periodic word by setting
	$$
	w_{(2k-1,k+1)}=a_k^{(n_{k+1})}.
	$$
	\item Define the last periodic word by setting
	$$
	w_{(2k,k+1)}=b_k^{(n_{k+1})}.
	$$
	\item Define the first type of density word by setting
	$$
	a_{k+1}=\mathcal{P}[w_{(1,k-1)},w_{(2,k-1)},\dots,w_{(2k-2,k-1)},a_k,b_k|a_k,n_{k+1}].
	$$
	\item
 \label{itemd:lastk+1}
Define the second type of density word by setting
	$$
	b_{k+1}=\mathcal{P}[w_{(1,k-1)},w_{(2,k-1)},\dots,w_{(2k-2,k-1)},a_k,b_k|b_k,n_{k+1}].
	$$
	\end{enumerate}
 We choose the parameter $n_{k+1}$ sufficiently large such that for every $m<k+1$ and every word $u\in\mathcal{W}_m\setminus\{a_m\}$, we have
\begin{equation}\label{eqd:fraction5}
\frac{\mathcal{N}(u,a_{k+1}^{(2)})}{|a_{k+1}^{(2)}|}<\frac{1}{|u|(2|u|-1)^d}\sum_{j=m}^k\varepsilon_j.
\end{equation}
When $m<k$, recall that $\mathcal{N}(u,a_k^{(2)})/|a_k^{(2)}|<(\sum_{j=m}^{k-1}\varepsilon_j)/|u|(2|u|-1)^d$ by~\eqref{eqd:fraction4}.
Therefore it is possible to satisfy~\eqref{eqd:fraction5} because, for $n_{k+1}$ large, the word $a_{k+1}$  primarily consists of concatenated copies of $a_k$ and so we can make the left hand side of~\eqref{eqd:fraction5} as close as we want to $(\sum_{j=m}^{k-1}\varepsilon_j)/|u|(2|u|-1)^d$; in particular we can make it less than the right hand side of~\eqref{eqd:fraction5}.  When $m=k$ it is possible to satisfy~\eqref{eqd:fraction5} because $u$ does not occur as a subword of $a_k^{(2)}$ and, again, $a_{k+1}$ is primarily made of concatenated copies of $a_k$.  Similarly we choose $n_{k+1}$  sufficiently large  such that if $u\in\mathcal{W}_m\setminus\{b_m\}$, then we have
\begin{equation}\label{eqd:fraction6}
\frac{\mathcal{N}(u,b_{k+1}^{(2)})}{|b_{k+1}^{(2)}|}<\frac{1}{|u|(2|u|-1)^d}\sum_{j=m}^k\varepsilon_j.
\end{equation}
For our next estimate, let $p_k$ be the index, in $\Z^d$, of the group of periods of the word $a_k^{(\infty)}$ defined by
$$
a_k^{(\infty)}(x_1,\dots,x_d):=a_k(x_1\pmod{n_k},\dots,x_d\pmod{n_k}).
$$
We finally require that $n_{k+1}$ is sufficiently large such that
\begin{equation}\label{eqd:inductive-bound}
\frac{|a_k|}{|a_{k+1}|}<\frac{1}{(4k-2)p_k}-\frac{1}{3^k\cdot|a_k|}
\end{equation}
This is possible as long as the right-hand side of the inequality is positive.  Note that $p_k\leq|a_k|$ and that $4k-2<3^k$ because $k\geq3$.  Note that all of the words constructed on level $k+1$ have the same domain: $\{1,2,\dots,\prod_{i=1}^{k+1}n_i\}^d$.
We call the words constructed in~\eqref{itemd:firstk+1} through~\eqref{itemd:lastk+1} the {\em level-$(k+1)$ words} and let $\mathcal{W}_{k+1}$ denote the set of all level-$(k+1)$ words.   Therefore, the only thing that is required to complete our inductive construction is to prove the following proposition.

\begin{proposition}\label{propd:levelkplus1}
For any choice of distinct words  $u,v\in\mathcal{W}_{k+1}$, the word $u$ does not occur as a subword of $v^{(2)}$.
\end{proposition}
\begin{proof}
The proof of this proposition is analogous to that of Proposition~\ref{propd:level2}.  Fixing some  $v\in\mathcal{W}_{k+1}$,  we argue that for all $u\in\mathcal{W}_{k+1}\setminus\{v\}$, the statement holds.  Again, we split the argument into cases depending on the type of the word $v$.

First consider when  $v\in\{w_{(i,k+1)}\colon1\leq i\leq2k-2\}$.   We consider two cases, depending on the choice of $u$. First suppose that $u\in\{w_{(2k-1,k+1)},w_{(2k,k+1)},a_{k+1},b_{k+1}\}$.  By Proposition~\ref{propd:level1}, neither $a_k$ nor $b_k$ occurs as a subword of $w_{(i,k)}^{(2)}$.  The word $v^{(2)}=w_{(i,k+1)}^{(2)}$ is the self-concatenation of a large number of copies of $w_{(i,k)}$ and so neither $a_k$ nor $b_k$ occurs as a subword of $v^{(2)}$.  But at least one of $a_k$ and $b_k$ occurs as a subword of $u$, and so $u$ cannot be a subword of $v^{(2)}$.
Next suppose $u=w_{(j,k+1)}$ where $j\in \{1,2,\dots,2k-2\}$ and $j\neq i$.  The word $u=w_{(j,k+1)}$ is not a subword of $v^{(2)}=w_{(i,k+1)}^{(2)}$ because $w_{(j,k)}$ is not a subword of $w_{(i,k)}^{(2)}$ and, as noted, $v^{(2)}$ is the self-concatenation of copies of $w_{(i,k)}$ while $w_{(j,k+1)}$ is the self-concatenation of copies of $w_{(j,k)}$.  Thus
$u$ does not occur as a subword of $v^{(2)}$ for any $u\in\mathcal{W}_{k+1}\setminus\{v\}$.  Thus the statement holds when $v\in\{w_{(i,k+1)}\colon1\leq i\leq2k-2\}$.

Next consider when $v=w_{(i,k+1)}$ for some $i\in\{2k-1,2k\}$.  Again we have two cases, depending on the choice of $u$.
First suppose $u\in\{w_{(j,k+1)}\colon1\leq j\leq2k-2\}\cup\{,a_{k+1},b_{k+1}\}$.
By Proposition~\ref{propd:level1}, none of the words in the set $\{w_{(j,k)}\colon1\leq j\leq2k-2\}$ occurs as a subword of $a_k^{(2)}$ or of $b_k^{(2)}$.  But by construction, the word $w_{(i,k+1)}^{(2)}$ is the self-concatenation of a large number of copies of $a_k$ or of $b_k$, whereas at least one of the words in $\{w_{(j,k)}\colon1\leq j\leq2k-2\}$ occurs as a subword of $u$.  Therefore $u$ occurs as a subword of $v^{(2)}$.  We next consider when $j\in\{2k-1,2k\}\setminus\{i\}$ and let $u=w_{(j,k+1)}$.  Then one of the words $u$ and $v$ is the self-concatenation of many copies of $a_k$ and the other is the self-concatenation of many copies of $b_k$.  By Proposition~\ref{propd:level1}, $a_k$ is not a subword of $b_k^{(2)}$ and $b_k$ is not a subword of $a_k^{(2)}$, and so $u$ is not a subword of $v^{(2)}$.  This establishes the statement when $v\in\{w_{(2k-1,k+1)},w_{(2k,k+1)}\}$.

Finally consider when $v\in\{a_{k+1},b_{k+1}\}$.  We give the argument when $v=a_{k+1}$, and the argument for $b_{k+1}$ is similar. Again, we have two cases, depending on the choice of $u$.
First suppose $u\in\mathcal{W}_{k+1}\setminus\{w_{(2k-1,k+1)},a_{k+1}\}$.  
By~\eqref{eqd:fraction5}, we know that $\mathcal{N}(x,a_{k+1}^{(2)})/|a_{k+1}^{(2)}|<\frac{1}{|x|(2|x|-1)^d}\sum_{j=m}^k$ for $x\in\mathcal{W}_k\setminus\{a_k\}$.  For any such $x$ and any particular occurrence of $x$ in $a_{k+1}^{(2)}$, there are $(2|x|-1)^d$ subwords of $a_{k+1}^{(2)}$ that have shape $|x|$ and partially (or completely) overlap this occurrence of $x$.  This means that, for any $x\in\mathcal{W}_k\setminus\{a_k\}$, if we look at the collection of all locations within $a_{k+1}^{(2)}$ where $x$ occurs, we have
$$
\frac{\text{$\#$ of subwords of shape $|x|$ in $a_{k+1}^{(2)}$ that overlap an occurrence of $x$}}{|a_{k+1}^{(2)}|}<\frac{1}{|x|}\sum_{j=m}^k.
$$
Recall that $p_k$ is the index, in $\Z^d$, of the stabilizer subgroup of $a_k^{(\infty)}\colon\Z^d\to\{0,1\}$ (when acted on by $\Z^d$ translations) and the number of occurrences of $a_k$ as a subword of $a_k^{(2)}$ satisfies~\eqref{eqd:multiplicity}.  Since all $2k-2$ words in $\mathcal{W}_k$ have the same size, we deduce from~\eqref{eqd:fraction1} that
\begin{equation}
\frac{\mathcal{N}(a_k,a_{k+1}^{(2)})}{|a_{k+1}^{(2)}|}\geq\frac{1}{p_k}-(2k-2)\cdot\frac{1}{|x|}\sum_{j=m}^k
\end{equation}
because $a_{k+1}^{(2)}$ is made by concatenating words in $\mathcal{W}_k$.  On the other hand, we claim that $\mathcal{N}(a_k,u)\leq(2k-1)\cdot(2|a_k|-1)^d$. To check this, note that if $u=b_{k+1}$, any occurrence of $a_k$ in $u$ must partially overlap elements of $\{w_{(i,k)}\colon1\leq i<2k-2\}\cup\{a_k\}$, in the postcard function definition of $b_{k+1}$, because it does not occur as a subword of $b_k^{(2)}$, and there are only $(2k-1)\cdot(2|a_k|-1)^d$ locations that have such overlaps.  If $u\neq b_{k+1}$ then $a_k$ does not occur in $u$, by Proposition~\ref{propd:level1} and the claim follows.

Now we return to showing that $u$ does not occur as a subword of $a_{k+1}^{(2)}$.  For contradiction, suppose $u$ does occur as a subword of $a_{k+1}^{(2)}$.  Recall, from the definition, that $a_{k+1}^{(2)}$ is primarily made of concatenated copies of $a_k$ and there are $2k-1$ specific locations where the words in the set $\{w_{(i,k)}\colon1\leq i<2k-2\}\cup\{b_k\}$ occur instead.  Since $u$ occurs as a subword of $a_{k+1}^{(2)}$ this means $u$ is primarily made of concatenated copies of $a_k$.  In fact, allowing for the possibility that $u$ occurs within $a_{k+1}^{(2)}$ in a location that overlaps the copies of $\{w_{(i,k)}\colon1\leq i<2k-2\}\cup\{b_k\}$ and that there may be no copies of $a_k$ that partially overlap these words, we can still use~\eqref{eqd:multiplicity} to estimate that
$$
\mathcal{N}(a_k,u)\geq((n_{k+1}-2)\prod_{i=2}^kn_i)^d\cdot\frac{1}{p_k}-(2k-1)\cdot(2n_k-1)^d.
$$
The reasoning for this is analogous to that of Level 3.

Lastly, suppose that $u=w_{(2k-3,k+1)}$.
By construction, $u$ is the self-concatenation of many copies of $a_k$.  If $u$ occurred within $v^{(2)}=a_{k+1}^{(2)}$ then it would either contain or partially overlap an occurrence of $w_{(1,k)}$ which is used in the postcard function definition of $a_{k+1}$.  Since $w_{(1,k)}$ does not occur as a subword of $a_k^{(2)}$, this occurrence of $u$ cannot contain $w_{(1,k)}$ as a subword so it can only occur in a location that partially overlaps $w_{(1,k)}$.  But $|u|=|v|$ and $u$ is the periodic concatenation of copies of $a_k$, so this occurrence of $u$ overlaps the copy of $w_{(1,k)}$ in all $2^d$ copies of $a_{k+1}$ that appear in the definition of $a_{k+1}^{(2)}$.  By periodicity of $u$, $w_{(1,k)}$ occurs as a subword of $a_k^{(2)}$, which is impossible.
This establishes the statement when $v=a_{k+1}$.  The argument for $v=b_{k+1}$ is similar, with the roles played by $a_k$ and $b_k$ switched.
\end{proof}

\subsubsection{Summarizing the construction and its properties}
The remainder of the construction is parallel to that given in the case when $d=1$. We have built a unique function ${\bf x}\colon \Z^d\to\{0,1\}$ whose restriction to the cube
$\{1-\prod_{i=1}^kn_i,\prod_{i=1}^kn_i\}^d$ is $a_k^{(2)}$,
and the $\mathbb{Z}^{d}$-subshift we define is the orbit closure of this function in $\{0,1\}^{\Z^d}$.  The proof that this system is a weakly CAM system is parallel to that given when $d=1$: any element of this system is either periodic (meaning has finite orbit under the action of $\Z^d$) or contains $a_k$ as a subword for all $k$, and so has dense orbit.  The fact that it is CAM again follows from Proposition~\ref{prop:weak-to-strong} because the system is expansive and weakly CAM.  The argument that this system has two non-atomic, ergodic measures is again parallel to that for $d=1$:
passing to a weak*-convergent subsequence of the empirical measures associated to the function ${\bf x}$,  and doing the same with the analogous function $\{1-\prod_{i=1}^kn_i,\prod_{i=1}^kn_i\}^d$ is $b_k^{(2)}$.  As in the case when $d=1$, almost every ergodic component of each of these measures are nonatomic and these measures are distinct because they give different measure to the cylinder set $\{x\in X\colon x(0,0,\dots,0)=0\}$.

\section{Further directions}
\label{sec:further}

We conclude with several questions about properties of CAM systems.
\begin{question}
    What are the automorphisms of a CAM $G$-system $(X, T)$, meaning what are all homeomorphsisms $S\colon X\to X$ such that $S\circ T_{g} = T_{g} \circ S$ for all $g \in G$?
\end{question}
As an example, for the $\Z$-system defined in~\ref{sec:def-of-Z-system},   we can easily check that the  automorphism group of this particular system contains the bit flip map (interchanging the letters $0$ and $1$).  However, we have little further insight into the automorphism group and its properties for this system or, more generally, whether there exist nontrivial constraints on the automorphism group of a CAM system.

\begin{question}
    What are the possible entropies for a CAM system? What are the possible growth rates of complexity for CAM $\mathbb{Z}^{k}$-subshifts?
\end{question}
In particular, this question is of interest for $\Z$-CAM systems and we do not know the entropy of the system defined in~\ref{sec:def-of-Z-system}.
This leads to a related question about chaotic systems:
\begin{question}
    Does every chaotic system with positive entropy have a proper infinite subsystem?
\end{question}
If there exists a  positive entropy CAM system, then the answer is no, in contrast to the behavior along subsequences as shown in~\cite{HLY}.  However, we do not know if any of our examples have positive entropy.

If a group $G$ admits a faithful action with dense periodic points, then we show in Proposition~\ref{prop:residually} that the group is necessarily residually finite. It seems plausible that this condition is in fact sufficient, and so we propose the following:
\begin{conjecture}
\label{conj:residual}    Every residually finite group is CAM.
\end{conjecture}

In Proposition~\ref{prop:weak-to-strong}, we show that
an expansive and weakly CAM system is a CAM system.  However the converse we show is only that a CAM system is weakly CAM, and we ask if this can be improved:
\begin{question}
\label{question:expansive}
    Is every CAM system expansive?
\end{question}
Note that it is shown in~\cite{BBCDS, GW}  that any $\Z$-CAM system satisfies sensitivity to initial conditions. Still, we suspect that the answer to the above question is no, but do not have a counterexample.

\end{document}